\documentclass[12pt]{article}
\usepackage{authblk}

\usepackage{setspace}
\usepackage{amsmath}

\usepackage{graphicx}
\usepackage{fullpage}
\usepackage{microtype}
\usepackage{amsmath,amsthm,amssymb}
\usepackage{color}
\usepackage{verbatim}
\usepackage{float}
\usepackage{framed}
\usepackage{tikz-cd}
\usetikzlibrary{arrows,automata,positioning}
\usepackage[normalem]{ulem}

\usepackage{comment}
\usepackage{algorithm}
\usepackage{algpseudocode}
\usepackage{url}
\usepackage{hyperref}
\usepackage{chngcntr}
\usepackage{apptools}
\usepackage{mathtools}
\usepackage{float}

\usepackage{framed}

\usepackage{pgfplots}
\usetikzlibrary{calc}
\pgfplotsset{my style/.append style={axis x line= center, axis y line=
center, axis equal }}

\usepackage[margin=1in]{geometry}
\definecolor{darkred}{rgb}{0.8, 0.0, 0.0}

\newcommand{\dis}{\mathrm{dis}}
\newcommand{\diam}{\mathrm{diam}}

\newcommand{\dgh}{d_{\mathrm{GH}}}
\newcommand{\di}{d_{\mathrm{I}}}
\newcommand{\db}{d_{\mathcal{B}}}
\newcommand{\cost}{\mathrm{cost}}
\newcommand{\codom}{\mathrm{codom}}
\newcommand{\dom}{\mathrm{dom}}
\newcommand{\low}{\mathrm{low}}

\newcommand{\N}{\mathbb{N}}

\newcommand{\eps}{\epsilon}

\newcommand{\V}{\mathcal{V}}
\newcommand{\K}{\mathrm{K}}
\newcommand{\M}{\mathcal{M}}
\newcommand{\F}{\mathbb{F}}
\newcommand{\C}{\mathfrak{C}}
\newcommand{\D}{\mathcal{D}}
\newcommand{\G}{\mathcal{G}}
\newcommand{\Z}{\mathcal{Z}}
\newcommand{\B}{\mathcal{B}}
\newcommand{\Vb}{\mathbb{V}}
\newcommand{\Wb}{\mathbb{W}}
\newcommand{\Lm}{\mathrm{L}}
\newcommand{\Sc}{\mathcal{S}}
\newcommand{\I}{\mathbb{I}}
\newcommand{\dgm}{\mathrm{dgm}}
\newcommand{\kvr}{\K^{\mathrm{VR}}}
\newcommand{\lc}{\{ \! \{ }
\newcommand{\rc}{\} \! \}}
\newcommand{\vr}{\mathrm{VR}}
\newcommand{\mor}{\mathrm{Mor}}
\newcommand{\Cc}{\mathcal{C}}
\newcommand{\pc}{\mathcal{P}}
\newcommand{\Span}{\mathrm{span}}
\newcommand{\T}{\mathbf{T}}

\newtheorem{theorem}{Theorem}[section]
\newtheorem{lemma}[theorem]{Lemma}
\newtheorem{proposition}[theorem]{Proposition}
\newtheorem{claim}[theorem]{Claim}
\newtheorem{corollary}[theorem]{Corollary}
\newtheorem{definition}[theorem]{Definition}

\newtheorem{fact}[theorem]{Fact}

\newtheoremstyle{ex}%
{}{}%
{\slshape}{}%
{\bfseries}{}{0.5em}{}%

\theoremstyle{ex}
\newtheorem{example}[theorem]{Example}

\begin{document}
\title{A Primer on Persistent Homology of Finite Metric Spaces}

\author[1]{Facundo M\'emoli}
\author[2]{Kritika Singhal}

	\affil[1]{Department of Mathematics and Department of Computer Science and Engineering, 
		The Ohio State University.\thanks{\texttt{memoli@math.osu.edu}}}

\affil[2]{Department of Mathematics, 
		The Ohio State University.\thanks{\texttt{singhal.53@osu.edu}}}

\maketitle

\section{Introduction}
TDA (topological data analysis) is a relatively new area of research related to importing classical ideas from topology into the realm of data analysis. Under the umbrella term TDA, there falls, in particular, the notion of \emph{persistent homology} PH, which can be described in a nutshell, as the study of scale dependent homological invariants of datasets.

The so called ``persistent homology pipeline" is depicted in Figure \ref{fig:pipe}: datasets are modeled as finite metric spaces. A given finite metric space induces a \emph{filtered simplicial complex} (via the Vietoris-Rips construction), which in turn, via the homology functor induces a \emph{persistence vector space}. Finally, these persistence vector spaces are decomposed into certain building blocks which give rise to a \emph{persistence diagram}. The figure  suggests that if two different datasets  (modeled as finite metric spaces) $(X,d_X)$ and $(Y,d_Y)$ are given, the dissimilarity between them controls how dissimilar their persistence diagrams will be. In other words, the assignment of persistence diagram to a dataset is continuous (actually Lipschitz) in a suitable sense.

In these notes, we provide a terse self contained description of the main ideas behind the construction of persistent homology as an invariant feature of datasets, and also discuss details about its stability to perturbations. These notes also include a brief discussion about applications to biological data and an overview of software packages that implement the PH pipeline.

Other useful resources for a more in depth study of the different ideas contained in these notes are \cite{comp-top,carlsson-topo,ghrist}.

\begin{figure}
\centering
\scalebox{1}{\includegraphics{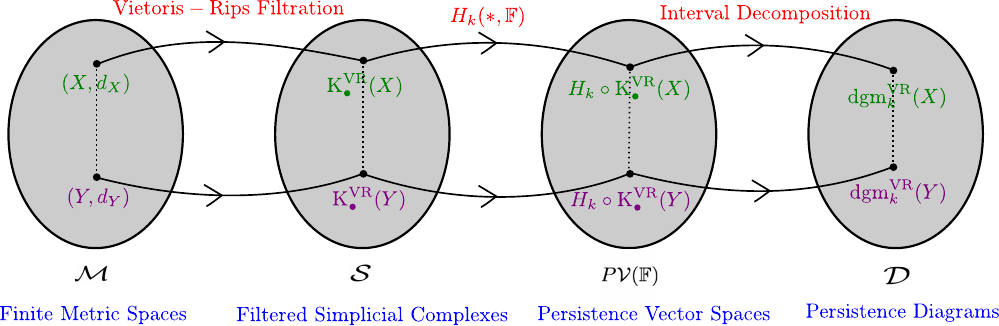}}
\caption{The persistent homology pipeline} \label{fig:pipe}
\end{figure}

\paragraph{Organization.} 
In Section \ref{sec:clustering} we provide a mathematical formulation of clustering (in both its flat and hierarchical forms) of finite metric spaces as a precursor for the notion of persistent homology.

In Section \ref{sec:simplicial} we cover the basics of simplicial homology -- a necessary ingredient for later discussing the theoretical elements pertaining to persistent homology.

In Section \ref{sec:persistence} we describe the persistent homology pipeline in detail, and in particular we review the construction of Vietoris-Rips persistence barcodes. In Section \ref{subsec:interpretation} we provide an analysis of the Vietoris-Rips barcodes corresponding to zero-dimensional persistent homology.

In Section \ref{sec:stability} we review the main theoretical elements underpinning the stability of Vietoris-Rips persistent homology of finite metric spaces.

In Section \ref{sec:applications} we overview a number of applications of persistent homology to biological data and beyond.

Finally, Section \ref{sec:software} provides a list of software packages implement different parts of the persistent homology pipeline.

\paragraph{Acknowledgements.}
These notes are meant to supplement the lectures given by the first author during the TGDA@OSU TRIPODS Summer School held at MBI during May 2018. Videos of the lectures are available at \cite{mbi}. We acknowledge NSF support  through project CCF \#1740761.

\tableofcontents

\section{Clustering} \label{sec:clustering}

One of the methods for extracting information from a data set is clustering the data set according to some rule. In this paper, datasets are represented as finite metric spaces. A finite metric space is a pair $(X,d_X)$, where $X$ is a finite set and $d_X:X \times X \rightarrow \mathbb{R}_+$ is a distance function. We denote by $\mathcal{M}$ the collection of all finite metric spaces. 

We start by providing a definition of a clustering method with some examples. For any $n \in \N$, we denote the set $\{1,2, \ldots,n \}$ by $[1:n]$. Given $(X,d_X) \in \mathcal{M}$, we denote by $P(X)$, the collection of all  partitions of $X$. Precisely, every $P \in P(X)$ is a family of sets $P = \{B_1, B_2, \ldots, B_k \} $, $k \leq |X|$, such that $B_i \subseteq X$ for all $i \in [1:k]$, for all $i,j \in [1:k]$ with $i \neq j$, $B_i \cap B_j = \emptyset$ and $\cup_{i=1}^k B_i = X$. We refer to each $B_i$, $i \in [1:k]$ as a \emph{block} of $P$. We denote by $\mathcal{P}$, the collection of all pairs $(X,P_X)$, where $X \in \mathcal{M}$ and $P_X \in P(X)$. Formally,
$$\mathcal{P} := \{ (X,P_X)~|~X \in \mathcal{M}, P_X \in P(X) \}.$$ 

\begin{definition}[Clustering Method]
A clustering method $\C$ is a map $\C: \mathcal{M} \rightarrow \mathcal{P}$ such that for every $(X,d_X) \in \mathcal{M}$, $\C((X,d_X)) = (X,P_X)$, where $P_X \in P(X)$.
\end{definition}

\begin{example}
An example of a clustering method is the discrete clustering that partitions every metric space into singletons. Precisely, we have $\C_{\mathrm{disc}} : \M \rightarrow \mathcal{P}$ with $\C_{\mathrm{disc}}((X,d_X)) = (X,S_X)$, where $S_X \in P(X)$ is the partition of $X$ into singletons. 
\end{example}

\begin{example}
Another example of a clustering method is the full clustering that partitions every metric space into a single block. Precisely, we have $\C_{\mathrm{full}}:\M \rightarrow \mathcal{P}$ with $\C_{\mathrm{full}}((X,d_X)) = (X,\{X \})$.
\end{example}

There are various other examples of clustering methods such as partitioning into clusters whose diameter is bounded above by a constant, or partitioning into clusters whose diameter is bounded below by a constant, and so on \cite{jardinebook}. Since we are working with finite metric spaces, the metric structure is the only information we have for determining a partition. Thus, it seems natural that for $(X,d_X), (Y,d_Y) \in \M$ and a structure preserving map $f:X \rightarrow Y$, a partition of $Y$ induced by a clustering method $\C$ can be determined, at least partially, using the map $f$ and a partition of $X$ induced by the same clustering method $\C$. Precisely, we want a clustering method $\C$ to be a \emph{functor}, see \cite{CM13}.

In order to view a clustering method $\C$ as a functor, we need to view $\M$ and $\mathcal{P}$ as categories. We refer the readers to \cite{jacobson2012basic,spivak2014category} for an account on category theory. We define the categorical structure on $\M$ and $\mathcal{P}$ as follows: 

\begin{definition}[Category of Finite Metric Spaces]
Let $\mathcal{M}$, by abuse of notation, denote the category of finite metric spaces. The objects of $\mathcal{M}$ are finite metric spaces $(X,d_X)$, and the morphisms are defined as follows: for $(X,d_X),(Y,d_Y) \in \mathcal{M}$, we say that a set map $\phi:X \rightarrow Y$ belongs to $\mathrm{Mor}_{\mathcal{M}}((X,d_X),(Y,d_Y))$ if for all $x,x' \in X $,
$$d_X(x,x') \geq d_Y((\phi(x), \phi(x')).$$
In other words, $\phi$ is $1$-Lipschitz. 
\end{definition}

We observe that for all $(X,d_X), (Y,d_Y) \in \M$, the set $\mathrm{Mor}_\M((X,d_X),(Y,d_Y)) \neq \emptyset$, since the map $\phi : X \rightarrow Y$ that sends every point in $X$ to a single point in $Y$ belongs to $\mathrm{Mor}_\M((X,d_X),(Y,d_Y))$. We now define the category of partitions of finite sets.

\begin{definition}[Category of Partitions of Finite Sets ]
Let $\mathcal{P}$, by abuse of notation, denote the category of partitions of finite sets. The objects of $\mathcal{P}$ are $(X, P_X)$, where $X$ is a finite set and $P_X \in P(X)$. Here, recall that $P(X)$ is the family of all partitions of $X$.

Given $P_Y = \{B_1, \ldots, B_k \} \in P(Y)$, and a set map $\phi:X \rightarrow Y$, the pullback of $P_Y$ along $\phi$ is defined as $\phi^\ast P_Y = \{\phi^{-1}(B_i)~|~i \in [1:k] \}$. Clearly, $\phi^\ast P_Y \in P(X)$. The morphisms in $\mathcal{P}$ are then defined as follows: for $(X,P_X),(Y,P_Y) \in \mathcal{P}$, we say that a set map $\phi:X \rightarrow Y$ belongs to $ \mathrm{Mor}_{\mathcal{P}}((X,P_X),(Y,P_Y))$ if $P_X$ is finer than $\phi^\ast P_Y$. This means that for every set $A \in P_X $, there exists a set $B \in \phi^\ast P_Y$ such that $A \subseteq B$.
\end{definition}

We observe that for all $(X,P_X), (Y,P_Y) \in \mathcal{P}$, the set $\mathrm{Mor}_\mathcal{P}((X,P_X),(Y,P_Y)) \neq \emptyset$, since the map $\phi:X \rightarrow Y$ that sends every point of $X$ to a single point of $Y$ satisfies $\phi^\ast P_Y = \{X\} $. Thus, any $P_X \in P(X)$ is finer than $\phi^\ast P_Y$, and we obtain $\phi \in \mathrm{Mor}_\mathcal{P}((X,P_X),(Y,P_Y))$.

We now define a clustering method $\C : \M \rightarrow \mathcal{P}$ to be a functor. This means that for all $(X,d_X), (Y,d_Y) \in \M$ and $\phi \in \mathrm{Mor}_\M((X,d_X),(Y,d_Y))$, 
$$\C((X,d_X)) \in \mathcal{P}~\mbox{and}~\C(\phi) \in \mathrm{Mor}_{\mathcal{P}}(\C((X,d_X)), \C((Y,d_Y))).$$ 
Furthermore, $\C$ satisfies $\C(\mathrm{id}_{(X,d_X)}) = \mathrm{id}_{\C((X,d_X))}$ and for all $(Z,d_Z) \in \M$ with $\psi \in \mathrm{Mor}_\M((Y,d_Y),(Z,d_Z)) $,
$$ \C(\psi \circ \phi) = \C(\psi) \circ \C(\phi).$$

We recall the clustering method $\C_{\mathrm{disc}}$ and show that $\C_{\mathrm{disc}}$ is a functor. For all $(X,d_X) \in \M$, $\C_{\mathrm{disc}}((X,d_X)) = (X,S_X)$, where $S_X \in P(X)$ is the partition into singletons. For any $(Y,d_Y) \in \M$, let $\phi:X \rightarrow Y$ be a set map such that $\phi \in \mathrm{Mor}_\M((X,d_X), (Y,d_Y))$. Then, clearly, $\phi \in \mathrm{Mor}_{\mathcal{P}}((X,S_X),(Y,S_Y))$ since $S_X$ is the partition of $X$ into singletons, and thus is finer than any other partition of $X$, in particular, $S_X$ is finer than $\phi^\ast S_Y$. It is trivial to check that $\C_{\mathrm{disc}}$ satisfies other properties of being a functor.

Similarly, it can be checked that the clustering method $\C_{\mathrm{full}}$ is a functor. We now provide another example of a clustering method that is also a functor, and is defined for every real number $\delta \geq 0$. It is called the Vietoris-Rips clustering functor.

\begin{example}[Vietoris-Rips clustering functor] \label{ex:VR}
The Vietoris-Rips clustering functor at a fixed scale parameter $\delta > 0$, is denoted by $\C^{\vr_\delta} $, and is defined as follows: given $(X,d_X) \in \mathcal{M}$ and $\delta > 0$, define $P_X(\delta) \in P(X)$ as $P_X(\delta) = X \slash \sim_\delta$, where $x \sim_\delta x'$ if and only if there exists a sequence $x_0, x_1, \ldots, x_n$ in $X$ with $x_0 = x $ and $x_n = x'$, such that for all $i \in [1:n]$, $d_X(x_{i-1}, x_i) \leq \delta$. Then, $\C^{\vr_\delta}((X,d_X)) :=  (X,P_X(\delta))$. The clustering $P_X(\delta) $ is referred to as the single linkage clustering of $X$ at scale $\delta $. 

Consider a metric space $(X,d_X)$ where $X = \{a,b \}$ and $d_X(a,b) = r > 0$. Then, for all $\delta < r$, $\C^{\vr_\delta}((X,d_X)) = \{ \{ a\}, \{ b\} \}$, and for $\delta \geq r$, $\C^{\vr_\delta}((X,d_X)) = \{a,b \}$.
\end{example}

The functoriality of $\C^{\vr_\delta}$ can be seen as follows: given $(X,d_X),(Y,d_Y) \in \M$, let $\phi \in \mor_\M((X,d_X),(Y,d_Y))$. Then, by definition, for all $x,x' \in X$, $d_X(x,x') \geq d_Y(\phi(x), \phi(x'))$. Now, let $\delta > 0$ be fixed. If $x,x' \in X$ are such that $x \sim_\delta x'$, then there exists a sequence $x = x_0, x_1, \ldots, x_n = x'$ in $X$ such that for all $0 \leq i \leq n-1$, $d_X(x_i, x_{i+1}) \leq \delta$. By definition of $\phi$, this implies that for all $0 \leq i \leq n-1$, $d_Y(\phi(x_i), \phi(x_{i+1})) \leq \delta$, and therefore $\phi(x) \sim_\delta \phi(x')$. Thus, we obtain that $P_X(\delta)$ is finer than $\phi^\ast P_Y(\delta)$. We can similarly check that $\C^{\vr_\delta} $ satisfies other properties of being a functor. 

Given $\delta \geq 0$, let $\Delta_2(\delta)$ denote the metric space consisting of $2$ points at distance $\delta$. The next theorem states the  uniqueness of the Vietoris-Rips clustering functor with respect to a particular property.

\begin{theorem}[{\cite[Theorem 6.4]{CM13}}] 
Let $\C : \M \rightarrow \mathcal{P}$ be a clustering functor for which there exists $\delta_\C > 0$ with the property that:
\begin{enumerate}
\item $\C(\Delta_2(\delta))$ is in two pieces for all $\delta \in [0, \delta_\C)$, and
\item $\C(\Delta_2(\delta))$ is in one piece for all $\delta \geq \delta_\C$.
\end{enumerate}
Then, $\C$ is the Vietoris-Rips clustering functor with parameter $\delta_\C$.
\end{theorem}

As we discussed, we have that $\C_{\mathrm{disc}}, \C_{\mathrm{full}}$ and the Vietoris-Rips clustering functor are   examples of functorial clustering methods. It is worth pointing out that the well known  average linkage and  complete linkage clustering methods fail to be functorial, see \cite{CM10}.

We observe that the Vietoris-Rips clustering functor $\C^{\vr_\delta} $ varies with $\delta$. Thus, a natural question one may ask is how the clustering at scale $\delta$ is related to the clustering at scale $\delta' \neq \delta$. This leads to the concept of hierarchical clustering.

\subsection{Hierarchical Clustering} \label{subsec:hierarchical}

We start by looking at an example. Consider a metric space $(Z,d_Z)$ where $Z = \{a,b,c \}$ and $d_Z(a,b) = 0.4, d_Z(b,c) = 0.6, d_Z(a,c) = 0.7$. Then, we have that for all $0 \leq \delta < 0.4$, $\C^{\vr_\delta}((Z,d_Z)) = \{ \{a\}, \{b \}, \{c \} \}$, for $0.4 \leq \delta < 0.6$, $\C^{\vr_\delta}((Z,d_Z)) = \{ \{a,b \}, \{c \} \}$ and for $\delta \geq 0.6 $, $\C^{\vr_\delta}((Z,d_Z)) = \{ \{a,b,c \} \}$. We observe that for $\delta = 0$, the clusters are singletons, and for $\delta$ large enough, all points fall into one cluster. In addition, for $\delta' > \delta$, the clusters at $\delta'$ are obtained by merging clusters at $\delta$. Such a clustering can be pictorially represented using a dendrogram.  

\begin{definition}[Dendrogram]
Let $X$ be a finite set. A dendrogram over $X$ is a function $\theta_X:[0, \infty) \rightarrow P(X)$, such that the following hold:
\begin{enumerate}
\item For all $s \leq t$, $\theta_X(s)$ is finer than $\theta_X(t)$.
\item $\theta_X(0)$ is the partition into singletons.
\item There exists $t_f \in (0, \infty)$ such that $\theta_X(t_f) = \{X \} $.
\item For all $t > 0$, there exists $\eps > 0$ such that $\theta_X(t + \eps) = \theta_X(t)$.
\end{enumerate}
The parameter $t$ is referred to as the scale of partition.
\end{definition}

A dendrogram depicting the Vietoris-Rips clustering (called the \emph{single linkage dendrogram}) of the $3$-point metric space $(Z,d_Z)$ described above is as follows: 
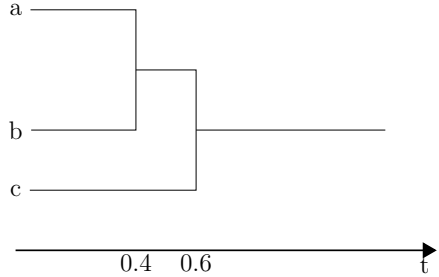
\begin{figure}[H]
\centering
\scalebox{0.8}{
\begin{tikzpicture}[sloped]
\node (a) at (-6,4) {a};
\node (b) at (-6,2) {b};
\node (c) at (-6,1) {c};
\node (ab) at (-4,3) {};
\node (all) at (-3,2) {};
\node (end) at (0,2) {};
\node (label) at (-4,-0.2) {0.4};
\node (lab) at (-3,-0.2) {0.6};
\node (labe) at (0.8,-0.25 ) {t};

\draw  (a) -| (ab.center);
\draw  (b) -| (ab.center);
\draw  (c) -| (all.center);
\draw  (ab.center) -| (all.center);
\draw  (all.center) -| (end);
\draw[->,-triangle 60,thick] (-6,0) --  (1,0);
\end{tikzpicture}}
\caption{Single linkage dendrogram of $(Z,d_Z)$.} \label{fig:sl}
\end{figure}
Here, we have that $\theta_Z(0.3) = \{\{a\}, \{b\}, \{c\} \} $, $\theta_Z(0.5) = \{\{a,b\}, \{c\} \} $ and $\theta_Z(1) = \{ \{a,b,c \} \} $. Precisely, we have that for every $t \geq 0$, the $\theta_Z(t) = \C^{\vr_t}((Z,d_Z))$.

Let $G(X)$ denote the collection of all dendrograms over a finite set $X$ and let $\mathcal{G} = \{(X,\theta_X)~|~|X| < \infty,~ \theta_X \in G(X) \} $. Then, $\mathcal{G}$ can be viewed as a category. The objects of $\G$ are as specified in the definition, and for all $(X,\theta_X), (Y, \theta_Y) \in \G$, a set map $\phi:X \rightarrow Y $ belongs to $\mathrm{Mor}_\G((X,\theta_X),(Y,\theta_Y)) $ if for all $t \geq 0$, $\theta_X(t) $ is finer than $\phi^\ast \theta_Y(t)$. We again have that for all $(X,\theta_X), (Y, \theta_Y) \in \G$, $\mathrm{Mor}_\G((X,\theta_X),(Y,\theta_Y)) \neq \emptyset$, since the map $\phi:X \rightarrow Y$ that takes all points of $X$ to a single point of $Y$ belongs to $\mathrm{Mor}_\G((X,\theta_X),(Y,\theta_Y))$. We are now ready to define hierarchical clustering formally.

\begin{definition}[Hierarchical Clustering]
A hierarchical clustering method is any functor $\mathcal{H} : \mathcal{M} \rightarrow \mathcal{G} $, i.e.~for any $(X,d_X) \in \M$, $\mathcal{H}((X,d_X)) = (X, \theta_X)$, where $\theta_X \in G(X)$.
\end{definition}

The Vietoris-Rips clustering functor, as described in Example \ref{ex:VR}, is a an example of a hierarchical clustering, since for any $X \in \M$, the function $P_X:[0, \infty) \rightarrow P(X)$, as defined in Example \ref{ex:VR}, is a dendrogram over $X$. The Vietoris-Rips clustering functor applied on $(X,d_X)$ induces a metric on $X$ in the following manner: given $x,x' \in X$, we can find the smallest $t > 0$, such that $x$ and $x'$ belong to same block of the partition $\C^{\vr_t}((X,d_X))$. This provides us with a measure of dissimilarity between points of $X$. This dissimilarity induces a metric on $X$, referred to as an ultra-metric. 

\begin{definition}[Ultra-metric]
Given a set $X$, a function $u : X \times X \rightarrow \mathbb{R}_{\geq 0} $ is called an ultra-metric if the following hold: 
\begin{enumerate}
\item For all $x,x' \in X$, $u(x,x') = u(x',x) \geq 0$ and $u(x,x') = 0$ if and only if $x=x'$. 
\item For all $x,x',x'' \in X$, $u(x,x'') \leq \max \{u(x,x'), u(x',x'') \} $.
The second condition is referred to as the strong triangle inequality.
\end{enumerate}
\end{definition}

We now define the ultra-metric induced by the Vietoris-Rips clustering functor.

\begin{definition}[Ultra-metric induced by $\C^\vr$] \label{def:ultrametric}
Let $(X,d_X) \in \M$. For every $x,x' \in X$, let $S_{x,x'}$ denote the collection of all sequences $x = x_0, x_1, \ldots, x_n = x'$in $X$ satisfying $x_i \neq x_j$ for all $i,j \in [1:n]$ with $i \neq j$. Then, the ultra-metric induced by $\C^\vr$, denoted by $u_X$, is defined as 
$$ u_X(x,x') := \min_{x=x_0, x_1, \ldots, x_n=x' \in S_{x,x'}} \max_{i \in [1:n]} d_X(x_i, x_{i-1}). $$
\end{definition}
It is straightforward to check that $u_X$ satisfies the properties of symmetry, positivity and strong triangle inequality. We observe that for any $x,x' \in X$, $u_X(x,x')$ is the smallest $t > 0$ at which the block containing $x$ merges with the block containing $x'$ in the single linkage dendrogram of $X$. Thus, we obtain that the Vietoris-Rips clustering functor applied to $(X,d_X)$ induces an ultra-metric on $X$. It has been shown in \cite[Theorem 18]{CM10} that the Vietoris-Rips clustering functor is the unique hierarchical clustering method with this property.

We remark that the ultrametric $u_X$ defined above is the \emph{maximal subdominant ultrametric on X}, which means that for every ultrametric $\hat{u}$ on $X$ satisfying $\hat{u} \leq d_X$, we have $\hat{u} \leq u_X$.

In subsequent sections, we describe the machinery of persistent homology --- a generalization of hierarchical clustering --- which can be used to obtain information about a metric space. The rest of the paper is focused on developing the theory of persistent homology. 

\section{Simplicial Homology} \label{sec:simplicial}

In this section, we define the pre-requisites needed to develop the theory of persistent homology. We will be defining and working only with abstract simplicial complexes in this paper. For the rest of the paper, any simplicial complex is an abstract simplicial complex. We refer the reader to \cite{munkresbook} for the definitions in this section.

\begin{definition}[Simplicial Complex]
A simplicial complex is a collection $\K$ of finite non-empty sets such that if $A$ is an element of $\K$, then so is every non-empty subset of $A$.
\end{definition}

For example, the collection $$\Lm = \{ \{3,5,7 \}, \{3,5 \}, \{3,7\}, \{5,7\}, \{2\}, \{3\}, \{5\}, \{7\} \}$$ forms a simplicial complex, but the collection $\{\{2,3\}, \{1\}, \{2\} \}$ does not.

\begin{definition}[Subcomplex]
Given a simplicial complex $\K$, a subcollection $\mathrm{J}$ of $\K$ is a subcomplex of $\K$ if $\mathrm{J}$ is a simplicial complex in itself.
\end{definition}

The collection $\{ \{3,5 \}, \{3 \}, \{5 \} \} $ is a subcomplex of the simplicial complex $\Lm$ defined above.

\begin{definition}[Simplex of a complex]
Every element $A$ of a simplicial complex $\K$ is a simplex of $\K$.
\end{definition}

Some of the simplices of $\Lm$ are $\{3,5,7\}$, $\{2\} $, and $\{5\}$. For every simplicial complex $\K$, if $\sigma = \{ x_0, x_1, \ldots, x_k \} $, $k \in \mathbb{N}$, is a simplex of $\K$, we assume that $\sigma$ is \emph{oriented} by the ordering $x_0 < x_1 < \ldots < x_k$. We write $[x_0, x_1, \ldots, x_k] $ to denote the equivalence class of the even permutations of this ordering, and $-[x_0,x_1,\ldots,x_k ] $ to denote the equivalence class of the odd permutations of this ordering. 

\begin{definition}[Face of a simplex]
The faces of a simplex $A$ of a simplicial complex $\K$ are the non-empty subsets of $A$.
\end{definition}

The faces of the simplex $[3,5,7 ]$ of the simplicial complex $\Lm$ defined above are 
$$\{3,5,7 \}, \{3,5 \}, \{3,7 \}, \{5,7 \}, \{3 \}, \{5 \}, \{ 7\}. $$

\begin{definition}[Dimension of a complex]
The dimension of a simplicial complex $\K$ is the largest dimension of a simplex of $\K$, where the dimension of a simplex $X$ of $\K$ is $|X| - 1$. If there is no such largest dimension, then dimension of $\K$ is infinite.
\end{definition}

The dimension of the simplicial complex $\Lm$ is $2$.

\begin{definition}[Vertices of a complex]
The vertex set of a simplicial complex $\K$, denoted by $V(\K)$, is the union of the one-point elements of $\K$.
\end{definition} 

The vertex set of the simplicial complex $\Lm$ is $\{2,3,5,7 \} $.

\begin{definition}[$n$-skeleton of a complex]
Given $n \in \mathbb{Z}_+$, an $n$-skeleton of a simplicial complex $\K$, denoted by $\K^S_n$, is the collection of all simplices of $\K$ of dimension at most $n$. 
\end{definition}

We observe that the $0$-skeleton of a simplicial complex $\K$ consists of all singletons of $\K$. For the simplicial complex $\Lm$, we have
$$ \Lm^S_0 = \{ [2], [3],[5], [7] \},~\Lm^S_1 = \{ [3,5], [3,7], [5,7] \} \cup \Lm^S_0~\mbox{and}~\Lm_2 = \{[3,5,7]\} \cup \Lm^S_1.  $$

\begin{definition}[Connected component]
Two simplices $S$ and $T$ of a simplicial complex $\K$ belong to the same connected component of $\K$ if there exists a non-empty sequence of simplices of $\K$, $S = S_0, S_1, \ldots, S_n = T$ such that for all $0 \leq i \leq n-1$, $S_i \cap S_{i+1} \neq \emptyset$.
\end{definition}

The connected components of the simplicial complex $\Lm$ are 

$$\mbox{$\{ [2] \}$ and $\{ [3,5,7], [3,5], [3,7], [5,7], [3], [5], [7] \} $.}$$

\begin{definition}[Simplicial map]
Given simplicial complexes $\K$ and $\K'$, a map $\phi: V(\K) \rightarrow V(\K')$ is called a simplicial map, if for every simplex $R$ of $\K$, $\phi(R)$ is a simplex of $\K'$. 
\end{definition}

The collection of all simplicial complexes along with simplicial maps between them forms a category. For any pair of simplicial complexes $\K$ and $\Lm$, a map that sends every vertex of $\K$ to the same vertex of $\Lm$ is a simplicial map. Thus, the set of simplicial maps between $\K$ and $\Lm$ is non-empty. We denote the category of simplicial complexes by $\Sc$. 

We need the following two definitions in order to define the homology groups.

\begin{definition}[Quotient vector space]
Let $\Vb$ and $\mathbb{U}$ be vector spaces over a field $\F$ such that $\mathbb{U} \subseteq \Vb$. We define an equivalence relation $\sim$ on $\Vb$ as follows: for $v,v' \in \Vb$, $v \sim v'$ if $v -v' \in \mathbb{U}$. For every $v \in \Vb$, the equivalence class of $v$ is denoted by $v + \mathbb{U}$, and is defined as $v + \mathbb{U} = \{v + u~|~u \in \mathbb{U}\} $. The quotient vector space $\frac{\Vb}{\mathbb{U}}$ is then defined as
$$ \frac{\Vb}{\mathbb{U}} := \Vb \slash \sim~:= \{v + \mathbb{U}~|~v \in \Vb\}.$$ 
\end{definition}

\begin{definition}[Isomorphic vector spaces]
Two vector spaces $\Vb$ and $\Wb$ over a field $\F$ are called isomorphic if there exists a bijective linear transformation $\phi : \Vb \rightarrow \Wb$.
\end{definition}

\begin{definition}[Chain Complex]
Let $\K$ be a simplicial complex and $\F$ be a field. Let $\K_n$ denote the collection of all simplices of $\K$ of dimension $n$. For every $n \in \mathbb{Z}_+$, we define 
$$C_n := \left\lbrace \sum_{i} c_i \sigma_i~|~c_i \in \mathbb{F}, \sigma_i \in \K_n \right\rbrace.$$
Precisely, $C_n$ is the free vector space over $\mathbb{F}$ with basis  $\K_n$. The boundary map $\partial_n:C_n \rightarrow C_{n-1}$ is defined as follows: for $\sigma = [x_0, x_1, \ldots, x_n] \in \K_n$, and $0 \leq i \leq n $, we denote the element $[x_0, x_1, \ldots, x_n ] \setminus x_i \in \K_{n-1} $ by $[x_0, x_1, \ldots, x_{i-1}, \widehat{x_i}, x_{i+1}, \ldots, x_n]$. Then, we set 
$$\partial_n(\sigma) := \sum_{i=0}^n [x_0, x_1, \ldots, x_{i-1}, \widehat{x_i}, x_{i+1}, \ldots, x_n] (-1)^i.$$ 
Since $C_n$ is a free vector space over $\K_n$, it suffices to define the boundary maps on elements of $\K_n$. The chain complex associated to $\K$, denoted by $\mathcal{C}_\ast(\K, \mathbb{F})$, is defined to be the sequence of vector spaces $\{C_n\}_{n \in \mathbb{Z}_+ }$, along with the boundary maps $\partial_n:C_n \rightarrow C_{n-1}$. Precisely, we have
$$ \Cc_\ast(\K,\F) := \ldots \xrightarrow{\partial_{n+1}} C_n \xrightarrow{\partial_n} C_{n-1} \xrightarrow{\partial_{n-1}} \ldots \xrightarrow{\partial_3} C_2 \xrightarrow{\partial_2} C_1 \xrightarrow{\partial_1} C_0 \xrightarrow{\partial_0} 0. $$
\end{definition}

\begin{lemma} \label{lem:chain}
Let $\K$ be a simplicial complex, and $\Cc_\ast(\K,\F)$ be the chain complex as defined above. Then, for all $n \in \mathbb{Z}_+$, we have $\partial_{n} \circ \partial_{n+1}  = 0$.
\end{lemma}

\begin{proof}
The lemma holds for any field $\F$, and the proof follows from the definition of boundary map. 
\end{proof}

Since $\partial_n \circ \partial_{n+1} = 0$ for all $n \in \mathbb{Z}_+$, we obtain that for every $n \in \mathbb{Z}_+$, the image of $\partial_{n+1} $ is contained in the kernel of $\partial_n $. 

\begin{definition}[$n$-Cycle and $n$-Boundary]
Given a simplicial complex $\K$ and its associated chain complex $\mathcal{C}_*(\K, \mathbb{F}) = (C_n, \partial_n)_{n \in \mathbb{Z}_+}$, the kernel of the map $\partial_n$ is the set of $n$-cycles and is denoted by $\mathcal{Z}_n(\K,\F)$. The image of the map $\partial_{n+1} $ is the set of $n$-boundaries, and is denoted by $\mathcal{B}_n(\K,\F)$. \end{definition}

By Lemma \ref{lem:chain}, we have that for all $n \in \mathbb{Z}_+$, $\mathcal{B}_n(\K,\F) \subseteq \mathcal{Z}_n(\K,\F)$.

\begin{definition}[Simplicial Homology]
Given $n \in \mathbb{Z}_+ $, the $n$-th homology group of a simplicial complex $\K$, is denoted by $H_n(\K, \mathbb{F})$, and is defined as 
$$H_n(K, \mathbb{F}) := \frac{\mathcal{Z}_n(\K,\F)}{\mathcal{B}_n(\K,\F)}. $$
That is, $H_n(\K, \mathbb{F})$ is a quotient vector space and the elements of $H_n(\K, \mathbb{F})$ are equivalence classes of $n$-cycles of $\Cc_\ast(\K,\F)$.
\end{definition}

\begin{definition}[Betti numbers]
Given $n \in \mathbb{Z}_+$, the $n$-th Betti number of a simplicial complex $\K$ is denoted by $\beta_n(\K)$, and is defined as $\beta_n(\K) := \dim(H_n(\K,\F))$.
\end{definition}

\begin{lemma} \label{lem:betti0}
For every simplicial complex $\K$, $\beta_0(\K)$ is equal to the number of connected components of $\K$.
\end{lemma}

\begin{proof}
For any simplicial complex $\K$, we have $H_0(\K,\F) = \frac{\mathcal{Z}_0(\K,\F)}{\mathcal{B}_0(\K,\F)}$, where $\Z_0(\K,\F) = \mathrm{ker}(\partial_0)$ and $\B_0(\K,\F) = \mathrm{im}(\partial_1)$. The $1$-simplex $\K_1$ consists of all elements of $\K$ of cardinality $2$, while the $0$-simplex $\K_0$ consists of singletons of $\K$. We use the symbol $\cong$ to denote an isomorphism of the concerned spaces, and $\langle S \rangle_F$ to denote the free vector space over $\F$ with basis elements of $S$. We have $\langle \K_0 \rangle_\F \cong \F^{|\K_0|} $, and $\langle \K_1 \rangle_\F \cong \F^{|\K_1|} $. The map $\partial_0 : \langle \K_0 \rangle_\F \rightarrow 0$ satisfies $\mathrm{ker}(\partial_0) = \langle \K_0 \rangle_\F \cong \F^{|\K_0|} $. The image under $\partial_1$ of an element $[x_i, x_j ] \in \K_1$ is $[x_j] - [x_i] \in \langle \K_0 \rangle_\F$. The elements $[x_i,x_j], [x_j,x_k] \in \K_1 $ belong to the same connected component, and span a subspace of dimension $2$ in $\langle \K_0 \rangle_\F$ with basis $\{ [x_j]-[x_i], [x_k]-[x_j] \} $. In general, if a connected component $S$ in $\K$ contains $n$ vertices, then the image under $\partial_1$ of the elements of $\K_1$ belonging to $S$ is a vector space of dimension $n-1$. Thus, if $\K$ has $l$ connected components $S_1, S_2, \ldots, S_l$, then we have that $\mathrm{im}(\partial_1) \cong \F^{\sum_i (|S_i|-1)}$. Thus, we obtain $H_0(\K,\F) \cong \frac{\F^{|\K_0|}}{\F^{\sum_i (|S_i|-1)}} $. Now, we know that $\K_0$ consists of all singletons and therefore $|\K_0|$ is equal to the number of vertices of $\K$. Every vertex of $\K$ belongs to a unique connected component, therefore we have that $\sum_i |S_i| = |\K_0|$. This implies that $H_0(\K,\F) \cong \F^l $, and we obtain that $\beta_0(\K)$ is equal to the number of connected components of $\K$. 
\end{proof}

Given simplicial complexes $\K$ and $\Lm$, and a simplicial map $\phi:\K \rightarrow \Lm$, a natural question to ask is whether $\phi$ induces a map between chain complexes $\Cc_\ast(\K,\F)$ and $\Cc_\ast(\Lm, \F)$, as well as between homology vector spaces $H_n(\K,\F)$ and $H_n(\Lm,\F)$, for $n \in \mathbb{Z}_+$. The following proposition answers this question.

\begin{proposition} \label{prop:functor}
Given simplicial complexes $\K$ and $\Lm$, a simplicial map $\phi:\K \rightarrow \Lm$ induces a map $\overline{\phi}: \Cc_\ast(\K,\F) \rightarrow \Cc_\ast(\Lm,\F) $, as well as maps $H_n(\phi,\F):H_n(\K,\F) \rightarrow H_n(\Lm,\F)$, for every $n \in \mathbb{Z}_+$. 
\end{proposition}

\begin{proof}
Let $\K$ and $\Lm$ be simplicial complexes and $\phi:\K \rightarrow \Lm$ be a simplicial map. Let $\Cc_\ast(\K,\F) = \{C_n(\K) \xrightarrow{\partial_n^\K} C_{n-1}(\K) \}_{n \in \mathbb{Z}_+} $ and $\Cc_\ast(\Lm,\F) = \{C_n(\Lm) \xrightarrow{\partial_n^\Lm} C_{n-1}(\Lm) \}_{n \in \mathbb{Z}_+} $. For every $n \in \mathbb{Z}_+$, $C_n(\K)$ and $C_n(\Lm)$ are free vector spaces over the collection of $n$-simplices of $\K$ and $\Lm$ respectively. Therefore, the map $\overline{\phi}:C_n(\K) \rightarrow C_n(\Lm) $ defined by linearly extending $\phi$ is a well-defined map. Precisely, we have that for $n \in \mathbb{Z}_+$, $i \in I$, $I$ an indexing set, $c_i \in \F$, $\sigma_i \in \K_n$ and $\sum_i c_i \sigma_i \in C_n(\K) $, $\overline{\phi}\left(\sum_i c_i \sigma_i \right) = \sum_i c_i \phi(\sigma_i)$. Since $\phi$ is a simplicial map, we obtain $\sum_i c_i \phi(\sigma_i) \in C_n(\Lm) $. Thus, we obtain the following diagram:
\[ 
\begin{tikzcd}
\ldots \arrow[r] & C_n(\K) \arrow[r,"\partial_n^\K "] \arrow[d,"\overline{\phi} "] & C_{n-1}(\K) \arrow[r,"\partial_{n-1}^\K "] \arrow[d,"\overline{\phi}"] & \ldots \arrow[r,"\partial_2^\K"] & C_1(\K) \arrow[r,"\partial_1^\K"] \arrow[d,"\overline{\phi}"] & C_0(\K) \arrow[r,"\partial_0^\K"] \arrow[d,"\overline{\phi}"] & 0 \\
\ldots \arrow[r] & C_n(\Lm) \arrow[r,"\partial_n^\Lm "] & C_{n-1}(\Lm) \arrow[r,"\partial_{n-1}^\Lm "] & \ldots \arrow[r,"\partial_2^\Lm"] & C_1(\Lm) \arrow[r,"\partial_1^\Lm"] & C_0(\Lm) \arrow[r,"\partial_0^\Lm"] & 0
\end{tikzcd}
\]

We now show that the squares in the above diagram commute. Let $\sigma = [x_0,x_1,\ldots,x_n] \in C_n(\K) $. Then $\partial_n^\K(\sigma) = \sum_{i=0}^n [x_0, x_1, \ldots, \widehat{x_i},\ldots,x_n](-1)^i $. We have 
$$\overline{\phi}(\partial_n^\K(\sigma)) = \sum_{i=0}^n [\phi(x_0),\phi(x_1), \ldots, \widehat{\phi(x_i)} \ldots, \phi(x_n)](-1)^i = \partial_n^\Lm(\phi(\sigma)) = \partial_n^\Lm(\overline{\phi}(\sigma)). $$
Thus, we have shown that $\overline{\phi} \circ \partial_n^\K = \partial_n^\Lm \circ \overline{\phi}$ for every $n \in \mathbb{Z}_+$. This implies that for every $n \in \mathbb{Z}_+$, the map $\overline{\phi}$ sends the kernel of $\partial_n^\K$ to the kernel of $\partial_n^\Lm$, and the image of $\partial_n^\K$ to the image of $\partial_n^\Lm$. Thus, for every $n \in \mathbb{Z}_+$, $\overline{\phi}$ sends $\Z_n(\K,\F)$ to $\Z_n(\Lm,\F)$ and $\B_n(\K,\F)$ to $\B_n(\Lm,\F)$. This provides us with the map $H_n(\phi,\F): H_n(\K,\F) \rightarrow H_n(\Lm,\F) $ defined as $H_n(\phi,\F)\left( \frac{\Z_n(\K,\F)}{\B_n(\K,\F)}  \right) = \frac{\overline{\phi}(\Z_n(\K,\F))}{\overline{\phi}(\B_n(\K,\F))} \subseteq \frac{\Z_n(\Lm,\F)}{\B_n(\Lm,\F)} = H_n(\Lm,\F) $. It is straightforward to check that for any simplicial complex $\K$, $H_n(\mathrm{id}_\K,\F) = \mathrm{id}_{H_n(\K,\F)} $, and for simplicial maps $\phi:\K \rightarrow \Lm $, $\psi:\Lm \rightarrow \mathcal{N}$, $H_n(\psi \circ \phi,\F) = H_n(\psi,\F) \circ H_n(\phi,\F)$.
\end{proof}

A direct corollary of the above theorem is the following.

\begin{corollary}
Let $\V_\F$ denote the category of finite dimensional vector spaces over the field $\F$ with linear transformations. Then, for every $n \in \mathbb{Z}_+$, $H_n(\ast,\F):\Sc \rightarrow \V_\F$ is a functor.
\end{corollary}

We now introduce the concept of contiguous simplicial maps which will be used crucially later.

\begin{definition}[Contiguous Simplicial Maps]
Given simplicial complexes $\K$ and $\Lm$, simplicial maps $f,g:\K \rightarrow \Lm $ are said to be contiguous if for every simplex $\sigma \in \K$, $f(\sigma) \cup g(\sigma)$ is a simplex in $\Lm$.
\end{definition}

The following lemma states that contiguous maps agree at the level of homology groups.

\begin{lemma}[\cite{munkresbook}] \label{lem:contig}
For all $k \in \mathbb{N}$, and simplicial complexes $\K, \Lm$, if maps $f,g:\K \rightarrow \Lm$ are contiguous, then $H_k(f) = H_k(g): H_k(\K,\F) \rightarrow H_k(\Lm,\F) $.
\end{lemma}

In the next section, we describe how to use the machinery developed in this section for studying data sets.

\section{Persistent Homology} \label{sec:persistence}

Persistent homology is a tool that is widely used for studying data sets. The persistent homology pipeline consists of four steps which are outlined below. We remark that some of the terminologies used below have not been defined yet. We will define these later in the section. The pipeline is introduced before so as to provide motivation for this section.

\begin{enumerate}
\item We start with a finite metric space $(X,d_X)$. We recall that every finite dataset can be viewed as a metric space by defining a measure of dissimilarity between its data points.
\item We assign a filtered simplicial complex to the metric space $(X,d_X)$. There are many methods for constructing filtered simplicial complexes from finite metric spaces. We will describe some of these methods in this section.
\item For every $n \in \mathbb{Z}_+$, we apply the homology functor $H_n(\ast,\F)$ to the filtered simplicial complex obtained in the last step. This produces persistence vector spaces.
\item For every persistence vector space obtained in the last step, we determine the persistence diagram associated with it.
\end{enumerate}

We will see that the persistence diagrams obtained in the end encode features of the input data set. We now provide missing details from the above pipeline.

\begin{definition}[Filtered simplicial complex]
A filtered simplicial complex is a sequence of simplicial complexes $\{\K_\delta \}_{\delta \in \mathbb{R}_+} $ such that for all $\delta \leq \delta' $, $\K_\delta \subseteq \K_{\delta'} $.
\end{definition}

We now see some examples of filtered simplicial complexes that can be constructed from a finite metric space $(X,d_X)$.
\begin{definition}[Vietoris-Rips Complex]
Given $(X,d_X) \in \mathcal{M}$ and $\delta \geq 0$, define
$$ \K_{\delta}^{\vr}(X) := \{ \sigma \subseteq X~|~\diam(\sigma) \leq \delta \} .$$
\end{definition}

It is straightforward to see that a Vietoris-Rips complex is a legitimate simplicial complex. In addition, we have that for $\delta \leq \delta'$, $\K_\delta^{\vr}(X) \subseteq \K_{\delta'}^{\vr}(X) $. This is because, for every $\sigma \in \K_\delta^{\vr}(X)$, $\diam(\sigma) \leq \delta \leq \delta'$, and therefore, $\sigma \in \K_{\delta'}^{\vr}(X)$. Thus, for every finite metric space $(X,d_X)$, $\K_\bullet^\vr(X) = \{\K_\delta^{\vr}(X) \xrightarrow{i^X_{\delta,\delta'}} \K_{\delta'}^{\vr}(X)\}_{\delta \leq \delta'} $ is a filtered simplicial complex. Here, $i^X_{\delta,\delta'} $ is the inclusion map. The next proposition follows from the definitions.

\begin{proposition}
Let $\delta \geq 0 $ be fixed. Then, $\K_{\delta}^{\vr} : \mathcal{M} \rightarrow \mathcal{S}$ is a functor.
\end{proposition}

Another example of a filtered simplicial complex is the \v{C}ech complex.

\begin{definition}[\v{C}ech Complex]
Given $(X,d_X) \in \mathcal{M}$ and $\delta > 0$, define
$$ \check{\mathrm{C}}_\delta(X) := \{ \sigma \subseteq X~|~\min_{x \in X} \max_{p \in \sigma} d_X(x,p) \leq \delta \}. $$
\end{definition}

It is an easy exercise to check that for every $(X,d_X) \in \M$, $\{\check{\mathrm{C}}_\delta(X) \}_{\delta \in \mathbb{R}_+} $ is a filtered simplicial complex. 

We have now explained the second step of the persistent homology pipeline. The third step is applying the homology map on a filtered simplicial complex to obtain a persistence vector space.

\begin{definition}[Persistence Vector Space\cite{carlsson_2014}] \label{def:pvs}
A persistence vector space $\Vb$ over a field $\F$ is a collection of vector spaces $\{V_\delta\}_{\delta \in \mathbb{R}_+} $ over $\F$ and $\F$-linear maps $\{V_\delta \xrightarrow{v_{\delta,\delta'}} V_{\delta'} \} $ with the following properties:
\begin{enumerate}
\item For all $\delta \geq 0$, the map $v_{\delta,\delta} : V_\delta \rightarrow V_\delta$ is the identity map on $V_\delta$.
\item For all $\delta'' \geq \delta' \geq \delta$, the following diagram commutes:
\[
\begin{tikzcd}
V_\delta \arrow[drur, "v_{\delta,\delta''}", bend right=50] \arrow[r,"v_{\delta,\delta'} "] & V_{\delta'} \arrow[r, "v_{\delta',\delta''} "] & V_{\delta''}
\end{tikzcd}
\]
Precisely, we have $v_{\delta,\delta''} = v_{\delta',\delta''} \circ v_{\delta,\delta'} $. 
\end{enumerate}
We use $ P\V(\F)$ to denote the collection of all persistence vector spaces over the field $\F$.
\end{definition}

Given a finite metric space $(X,d_X)$, a filtered simplicial complex $\{\K_\delta(X) \}_{\delta \in \mathbb{R}_+} $ and a sequence $0 \leq \delta_1 \leq \ldots \leq \delta_{n} $, for every $k \in \mathbb{Z}_+$ the sequence
$$ H_k(\K_{\delta_1}(X),\F) \rightarrow H_k(\K_{\delta_2}(X),\F) \rightarrow \ldots \rightarrow H_k(\K_{\delta_{n-1}}(X),\F) \rightarrow H_k(\K_{\delta_n}(X),\F)  $$
forms a persistence vector space (where the maps are induced by the simplicial inclusions), since $H_k(\ast,\F):\Sc \rightarrow \V_\F$ is a functor. 

\begin{definition}[Morphisms of Persistence Vector Spaces\cite{carlsson_2014}]
Given $\mathbb{V}, \mathbb{W} \in P\mathcal{V}(\mathbb{F}) $, a morphism $\alpha : \mathbb{V} \rightarrow \mathbb{W} $ is a collection of linear maps $\alpha_\delta:V_\delta \rightarrow W_\delta $, $\delta \in \mathbb{R}_+$, such that the following diagram commutes for every $\delta \leq \delta' $:
\[
\begin{tikzcd}
V_\delta \arrow[r, "v_{\delta,\delta'}"] \arrow[d, "\alpha_\delta"] 
& V_{\delta'} \arrow[d, "\alpha_{\delta'}"] \\
W_\delta \arrow[r, "w_{\delta, \delta'}"] 
& W_{\delta}
\end{tikzcd}
\]
We say that $\alpha:\mathbb{V} \rightarrow \mathbb{W} $ is an isomorphism if each $\alpha_\delta : V_\delta \rightarrow W_\delta$ is an isomorphism of vector spaces. In this case, we write $\Vb \cong \Wb $.
\end{definition}

It will be useful to consider persistence vector spaces of finite length (indexed by natural numbers). A persistence vector space of length $n\in\N$ is any sequence $\{ V_i \xrightarrow{v_{i,i+1}} V_{i+1} \}_{i \in [1:n-1]}$ of vector spaces over $\F$ and $\F$-linear maps. 
In analogy with Definition \ref{def:pvs}, here we assume that the map $v_{i,j} = v_{j-1,j}\circ v_{j-2,j-1}\circ\cdots v_{i,i+1}$ for all $1\leq i< j\leq n$ and $v_{i,i}=\mathrm{id}_{V_i}$ for all $i\in[1:n].$ For $n \in \mathbb{Z}_+$, let $P\V_n(\F) $ denote the collection of all persistence vector spaces over the field $\F$ of length $n$.

\begin{definition}[Sampling map] \label{def:samplingmap}
Let $\mathbb{V}$ be any persistence vector space. Given a finite set $A\subset \mathbb{R}_+$ with $|A|=n$, we write $A=\{\alpha_1<\alpha_2<\cdots \alpha_n\}$ and consider the $A$-\emph{sampling map}
$$\mathbf{S}(\cdot,A):P\V(\F)\rightarrow P\V_n(\F)$$
defined by $$\{V_\delta \xrightarrow{v_{\delta,\delta'}} V_{\delta'} \}  =\mathbb{V}\mapsto \mathbb{V}^A=\{V_1^A \xrightarrow{v^A_{1,2}} V_2^A \xrightarrow{v^A_{2,3}} \cdots \xrightarrow{v^A_{n-1,n}} V_n^A\};$$  where for each $i\in[1:n]$, $V_i^A:=V_{\alpha_i}$, and $v_{i,i+1}^A:=v_{\alpha_i,\alpha_{i+1}}.$
\end{definition}
We now concentrate on persistence vector spaces of length $n$ and describe a full invariant for those. An invariant of persistent modules is any map $\iota:P\V_n(\F) \rightarrow \mathcal{I}$ into some set $\mathcal{I}$ such that $\mathbb{V}\cong \mathbb{W}$ implies $\iota(\mathbb{V}) = \iota(\mathbb{W})$. An invariant $\iota$ is a \emph{full invariant} if $\iota(\mathbb{V}) = \iota(\mathbb{W})$ implies that $\mathbb{V}\cong \mathbb{W}$.

The full invariants of persistence vector spaces are called Persistence Diagrams and will help us associate an algebraic signature to finite metric spaces.

\paragraph{Persistence diagrams of persistence vector spaces of length $n$.}

We now assume that for every $\Vb \in P\V_n(\F)$ and for every $i \in [1:n]$, $\dim(V_i) < \infty$. Thus, $P\V_n(\F)$ is the collection of all \emph{pointwise finite dimensional (pfd) persistence vector spaces} of length $n$, $n \in \mathbb{Z}_+$. The reason behind this assumption is that such persistence vector spaces have a simple representation in terms of interval persistence vector spaces.

In the same way that finite dimensional vector spaces can be classified up to isomorphism by their dimension, finite length persistence vector spaces admit a classification based on certain finite multisets of points in the plane. In particular, it is not true that  persistence vector spaces can be classified by the sequence of dimensions.

\begin{example}
Assume that $\mathbb{V}\in P\V_n(\mathbb{F})$. Consider the vector $$\mathrm{dim}(\mathbb{V}):=\big(\dim(V_1),\dim(V_2),\ldots,\dim(V_n)\big).$$
We claim that there exists a natural number $n$ and $\mathbb{V},\mathbb{W}\in P\V_n(\mathbb{F})$ such that $\mathbb{V}\not\cong\mathbb{W}$ but $\dim(\mathbb{V})=\dim(\mathbb{W})$.
This can be seen from the following example: let $\Vb = \F \xrightarrow{v_{1,2}} \F^2 \xrightarrow{v_{2,3}} \F$, where $v_{1,2}(1_\F) = (1_\F,0_\F)$, and $v_{2,3}((1_\F,0_\F)) = v_{2,3}((0_\F,1_\F)) = 1_\F$. Let $\Wb = \F \xrightarrow{w_{1,2}} \F^2 \xrightarrow{w_{2,3}} \F$, where $w_{1,2}(1_\F) = (1_\F,0_\F), w_{2,3}((1_\F,0_\F)) = 0_\F$, and $w_{2,3}((0_\F,1_\F)) = 1_\F$. We have $\dim(\Vb) = \dim(\Wb)$. Suppose $\Vb$ and $\Wb$ are isomorphic. Then, for $i= 1,2,3 $, there exist isomorphisms $\alpha_i : V_i \rightarrow W_i $ such that all squares in the following diagram commute:
\[
\begin{tikzcd}
\F \arrow[r,"v_{1,2}"] \arrow[d,"\alpha_1"] & \F^2 \arrow[r,"v_{2,3}"] \arrow[d, "\alpha_2"] & \F \arrow[d,"\alpha_3"] \\
\F \arrow[r,"w_{1,2}"] & \F^2 \arrow[r,"w_{2,3}"] & \F 
\end{tikzcd}
\]
Let $\alpha_1(1_\F) = a \cdot 1_\F$, and $\alpha_3(1_\F) = b \cdot 1_\F$, where $0 \neq a,b \in \F$. Here, $a,b \neq 0$ because both $\alpha_1$ and $\alpha_3$ are isomorphisms. Now, we have $v_{1,3}(1_\F) = v_{2,3} \circ v_{1,2}(1_\F) = 1_\F$, and similarly $w_{1,3}(1_\F) = w_{2,3} \circ w_{1,2}(1_\F) = 0_\F$. Then, we obtain
$$ \alpha_3 \circ v_{1,3}(1_\F) = b \cdot 1_\F \neq 0_\F = w_{1,3} \circ \alpha_1(1_\F). $$
This contradicts the commutativity of all squares. Thus, we conclude that $\Vb$ and $\Wb$ are not isomorphic.
\end{example}

The above example shows that for $\Vb \in P\V_n(\F)$, $\dim(\Vb)$ is not a full invariant of $\Vb$. The construction of a full invariant of a persistence vector space requires a more subtle approach which depends on the notion of persistence diagrams of persistence vector spaces (see Corollary \ref{coro:full} below). 

\begin{definition}[Interval persistence vector space]
Given $n \in \mathbb{N}$ and $b,d \in [1:n],~b \leq d$, an interval persistence vector space is defined as follows: $V_i = 0$ for all $i < b$ and $i > d$, and $V_i = \F$ for all $b \leq i \leq d$. The map between the $0$-vector spaces, as well as maps $0 \rightarrow \F $ and $\F \rightarrow 0$ are specified to be the $0$-maps. The maps $\F \rightarrow \F$ are specified to be identity maps. Such a persistence vector space is denoted by $\I(b,d)$. Thus, we have
$$ \mathbb{I}(b,d) = 0 \rightarrow \ldots \rightarrow 0 \rightarrow \F \rightarrow \F \rightarrow \ldots \rightarrow \F \rightarrow 0 \ldots \rightarrow 0.  $$
\end{definition}

An example of an interval persistence vector space is $\I(2,3) = 0 \rightarrow \F \rightarrow \F \rightarrow 0$. We now have the following theorem.

\begin{theorem}[\cite{CB12}]
For every $\Vb \in P\V_n(\F)$, there exist intervals $[b_i,d_i]_{i \in I} $, $I$ an indexing set such that for every $i \in I$, $b_i, d_i \in [1:n] $, and $\Vb \cong \bigoplus_{i \in I}\I(b_i,d_i) $.
\end{theorem}

Furthermore, we have the following theorem.

\begin{theorem}[Krull-Remak-Schmidt-Azumaya \cite{azumaya_1950}] 
Let $\Vb \in P\V_n(\F)$ and $\mathbb{V} = \bigoplus_{i \in I} \I(b_i,d_i) = \bigoplus_{j \in J} \I(b_j,d_j) $ be two decompositions of $\mathbb{V} $ into interval persistence vector spaces. Then, $|I| = |J| $ and there exists a permutation $\pi \in S_N $, $N = |I|$ such that for all $i \in I$, there exists $j \in J$ satisfying $i = \pi(j)$.
\end{theorem}

A consequence of the above theorems is that for every $\Vb \in P\V_n(\F)$, if $\Vb = \bigoplus_{i \in I}\I(b_i,d_i)$, then the multiset $\lc(b_i,d_i) \rc_{i \in I} $ is a full invariant of $\Vb$. This multiset is called the \emph{persistence diagram} of $\Vb$, and this brings us to the fourth step of the persistent homology pipeline.

\begin{definition}[Persistence Diagram]
Let $n \in \mathbb{N}$ and $\Vb \in P\V_n(\F)$ be a persistence vector space. Let $\Vb = \bigoplus_{i \in I}\I(b_i,d_i) $. Then, the persistence diagram of $\Vb$, denoted by $\dgm(\Vb)$ is defined as the multiset of intervals $\lc (b_i,d_i) \rc_{i \in I} $. 
\end{definition}

\begin{corollary}\label{coro:full}
For any $\mathbb{V},\mathbb{W}\in P\V_n(\mathbb{F})$ it holds that $\mathbb{V}\cong \mathbb{W}$ if and only if $\mathrm{dgm}(\mathbb{V})=\mathrm{dgm}(\mathbb{W}).$
\end{corollary}

Let $\D$ denote the collection of all multisets $\lc(b_i,d_i)\rc_{i\in I}$, where $b_i\leq d_i$ are non-negative real numbers, and for $n \in \N$, let $\D_{[1:n]}$ denote the collection of all multisets $\lc (b_i,d_i) \rc_{i \in I}$, where $b_i \leq d_i$ and $ \{ b_i, d_i~|~i \in I \} \subseteq [1:n]$.

We note that for any $\Vb \in P\V_n(\F)$, $\dgm(\Vb)$ is a collection of points in $\mathbb{R}^2$. For $\dgm(\Vb) = \lc(b_i,d_i) \rc_{i \in I} $, the $b_i$'s are referred to as the \emph{birth times} and are represented on the x-axis, while $d_i$'s are referred to as the \emph{death times} and are represented on the y-axis. Since $b_i \leq d_i$ for all $i \in I$, the points of $\dgm(\Vb)$ lie on or above the $x=y$ line in $\mathbb{R}^2$. For example, consider $\Vb = \I(1,5) \oplus \I(3,4)$. Then, the persistence diagram of $\Vb$ is depicted in the following figure:

\begin{center}
\begin{tikzpicture}
\begin{axis}[my style, xtick={1,...,7}, ytick={0,1,...,6},
xmin=1, xmax=7, ymin=0, ymax=7, xlabel={$b(birth)$}, ylabel={$d(death)$}, xlabel near ticks, ylabel near ticks]
\addplot[domain=0:6]{x};
\addplot[mark=*, only marks] coordinates {(1,5) (3,4)};
\end{axis}
\end{tikzpicture}
\end{center}
Another way of depicting $\dgm(\Vb)$ is through \emph{barcodes}. The following diagram depicts the barcode of $\Vb = \I(1,5) \oplus \I(3,4)$.
\begin{center}
\begin{tikzpicture}
\begin{axis}[my style, xtick={1,...,7}, ytick={0},
xmin=1, xmax=7, ymin=0, ymax=7]
\addplot[domain = 1:5]{1};
\addplot[domain = 3:4]{2};
\end{axis}
\end{tikzpicture}
\end{center}

\paragraph{Vietoris-Rips persistence diagrams of finite metric spaces.}

Now, given $(X,d_X) \in \M$, we define the spectrum of $(X,d_X)$ as 
$$ \mathrm{spec}(X) := \{ d_X(x,x')~|~x,x' \in X \}. $$ Let $n=|\mathrm{spec}(X)|$ and write $\mathrm{spec}(X) = \{0=\delta_1 <\delta_2<\ldots<\delta_n = \diam(X)\}$.  
Given an integer $k\geq 0$ we consider the persistence vector space $\mathbb{V}$ of length $n$ defined as  
$$\mathbb{V}_k^X:=\mathbf{S}\big(H_k\circ\mathrm{K}_\bullet^\mathrm{VR}(X),\mathrm{spec}(X)\big).$$

Here, $\mathbf{S}$ is the sampling map as given in definition \ref{def:samplingmap}. Now we need a process that is in some sense dual to sampling. Given a finite set $\{\alpha_1<\cdots <\alpha_n\} = A\subset \mathbb{R}_+$ with $|A| = n$, we define a map $$\mathbf{T}(\cdot,A): \mathcal{D}_{[1:n]}\rightarrow \mathcal{D}$$
to be the function satisfying
\begin{enumerate}
    \item (Additivity) For all $D_1, D_2 \in \D_{[1:n]}$, $\mathbf{T}(D_1 \sqcup D_2,A) = \mathbf{T}(D_1,A) \sqcup \T(D_2,A)$.
    \item (Definition on atomic elements) We define
    \begin{enumerate}
        \item $\T(\lc (1,n) \rc,A) := \lc (\alpha_1, \infty) \rc $.
        \item For all $j \in [1:n]$, $\T(\lc (j,j) \rc,A ) := \lc (\alpha_j, \alpha_{j+1}) \rc$.
        \item For all $i,j \in [1:n]$ with $i<j$ and $(i,j) \neq (1,n)$, $\T(\lc (i,j) \rc,A) := \lc (\alpha_i, \alpha_j) \rc $.
    \end{enumerate}
\end{enumerate}

Note that these properties uniquely determine the map $\T(\cdot,A)$. We illustrate how the map $\T(\cdot,A)$ works via the following example.

\begin{example}
Let $n \geq 3$, and $A = \{\alpha_1 < \ldots < \alpha_n \} \subset \mathbb{R}_+$. Let $D = \lc (1,1), (1,2), (1,n), (2,3) \rc$. Then, we have that
\begin{align*}
    \T(D,A) &= \T(\lc (1,1) \rc,A) \sqcup \T(\lc (1,2) \rc,A) \sqcup \T(\lc (1,n) \rc,A) \sqcup \T(\lc (2,3),A \rc) \\
    &= \lc (\alpha_1,\alpha_2 \rc \sqcup \lc (\alpha_1, \alpha_2) \rc \sqcup \lc (\alpha_1, \infty) \rc \sqcup \lc (\alpha_2,\alpha_3) \rc \\
    &= \lc (\alpha_1,\alpha_2), (\alpha_1,\alpha_2), (\alpha_1, \infty), (\alpha_2,\alpha_3) \rc. 
\end{align*}
\end{example}

We now have the following definition.

\begin{definition}[$k$-th Vietoris-Rips Persistence Diagram]
Given $(X,d_X) \in \M $ and $ k \in \mathbb{N}$, the $k$-th Vietoris-Rips persistence diagram of $(X,d_X)$ is defined as the $$\mathrm{dgm}^\mathrm{VR}_k(X):=\mathbf{T}\big(\mathrm{dgm}(\mathbb{V}_k^X),\mathrm{spec}(X)\big).$$ 
\end{definition}

\begin{example}
We now provide an example to illustrate the definitions. Let $X = \{a,b \} $ with $d_X(a,b) = 1$. The filtered Vietoris-Rips simplicial complex of $X$ is as follows:
$$ \K_\delta^\vr(X) = \{ [a], [b] \}~\forall~\delta < 1~\mbox{and}~\K_\delta^\vr(X) = \{ [a,b ], [a], [b] \}~\forall~\delta \geq 1. $$
Let $\K^1 = \{ [a], [b] \} $ and $\K^2 = \{[a,b], [a], [b] \}$. The set of $0$-simplices of $\K^1$ is $\K^1_n = \{ [a], [b] \} $ and for $n > 0$, the set of $n$-simplices of $\K^1$ is $\K^1_n = \emptyset$. The set of $0$-simplices of $\K^2$ is $\K^2_0 = \{ [a], [b] \}$, the set of $1$-simplices of $\K^2$ is $\K^2_1 = \{[a,b] \} $, and, for all $n \geq 2$, the set of $n$-simplices of $\K^2$ is $\K^2_n = \emptyset$. Thus, we have that $H_0(\K^1,\F) = \F^2$ and $H_k(\K^1,\F) = 0 $ for all $k \geq 1$. Similarly,  $H_0(\K^2,\F) = \F$ and $H_k(\K^2,\F) = 0$ for all $ k \geq 1$. This implies that $$H_0(\K_\bullet^\vr(X),\F) = \F^2 \rightarrow \ldots \rightarrow \F \rightarrow \ldots, $$
with the transition $\F^2 \rightarrow \F$ occurring at $\delta = 1$, and the notation $V \rightarrow \ldots $ meaning that all vector spaces hidden in the dots are $V$. Furthermore, we have that $H_k(\K_\bullet^\vr(X),\F) = 0$ for all $k \geq 1$. Thus, we have that $\dgm_k^\vr(X) = \emptyset$ for all $k \geq 1$. We now calculate $\dgm_0^\vr(X)$, using the maps $\mathbf{S}$ and $\mathbf{T}$ defined above. 

We have that $\mathrm{spec}(X) = \{0,1 \}$, and therefore $\mathbf{S}(H_0 \circ \K_\bullet^\vr(X),\mathrm{spec}(X))$ is a persistence vector space of length $2$ given by
$$ \Vb_0^X = \F^2 \xrightarrow{v_{1,2}} \F^1, $$
with the map $v_{1,2}$ being defined as $v_{1,2}((1_\F,0_\F)) = v_{1,2}((0_\F,1_\F)) = 1_\F$.

Clearly, we have $$\F^2 \xrightarrow{v_{1,2}} \F^1 \cong \big(\F \rightarrow \F \big) \oplus \big(\F \rightarrow 0\big).$$
Thus, we obtain $\dgm(\Vb_0^X) = \lc (1,2),(1,1) \rc $. We now apply the map $\mathbf{T}(\cdot, \mathrm{spec}(X))$ to $\dgm(\Vb_0^X)$ in order to obtain $\dgm_0^\vr(X)$. By definition, we have
$$ \mathbf{T}\left( \dgm(\Vb_0^X), \mathrm{spec}(X) \right) = \lc (0,\infty), (0,1) \rc. $$
Thus, we obtain $\dgm_0^\vr(X) = \lc (0,\infty),(0,1) \rc $.
\end{example}

\begin{example} \label{eg:4point}
\begin{figure}
    \centering
    \scalebox{2.5}{
    \begin{tikzpicture}
    %\tikzset{vertex/.style = {size = 1.5em} }
    \node (a) at (0,0) {$\scalebox{0.5}{1}$};
    \node (b) at (0,1) {$\scalebox{0.5}{2}$};
    \node (c) at (1,1) {$\scalebox{0.5}{3}$};
    \node (d) at (1,0) {$\scalebox{0.5}{4}$};
    \path[-] (a) edge node[left,blue]{$\scalebox{0.5}{1}$} (b);
    \path[-] (a) edge node[below,blue]{$\scalebox{0.5}{1}$} (d);
    \path[-] (b) edge node[above,blue]{$\scalebox{0.5}{1}$} (c);
    \path[-] (c) edge node[right,blue]{$\scalebox{0.5}{1}$} (d);
    \path[dotted,-] (a) edge node[right,blue]{$\scalebox{0.5}{2}$} (c);
    \path[dotted,-] (b) edge node[left,blue]{$\scalebox{0.5}{2}$} (d);
    \end{tikzpicture}
    }
    \caption{The $4$-point metric space $(X,d_X)$ of Example \ref{eg:4point}.}
    \label{fig:example4point}
\end{figure}
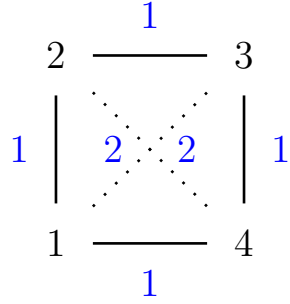
We now consider another example of a metric space with $4$ points, depicted in Figure \ref{fig:example4point}. This metric space is defined as follows:
$$ (X,d_X) = \left(\{1,2,3,4 \}, \begin{pmatrix}
0 & 1 & 2&1\\
1 &0&1&2 \\
2 & 1 & 0&1\\
1 &2 &1 & 0
\end{pmatrix}\right). $$
Thus, $X$ consists of the corners of a square of side length $1$, with $\ell_1$-distance. The filtered Vietoris-Rips simplicial complex of $X$ is as follows: 
$$\kvr_\delta(X) = \{[1],[2],[3],[4] \}~\forall~0 \leq \delta < 1,$$
$$ \kvr_\delta(X) = \{[1],[2],[3],[4],[1,2],[2,3],[3,4],[1,4] \}~\forall~1 \leq \delta < 2,$$
\begin{align*}
\kvr_\delta(X) =& \{[1],[2],[3],[4],[1,2],[2,3],[3,4],[1,4],[1,3],[2,4],[1,2,3],[1,2,4], \\
&[1,3,4],[2,3,4], [1,2,3,4] \}~\forall~\delta \geq 2.
\end{align*}
Let $\K^1 := \kvr_0(X)$, $\K^2 := \kvr_1(X)$, and $\K^3 = \kvr_2(X)$. For $i = \{1,2,3 \} $ and $j \in \mathbb{Z}_+$, let $\K^i_j$ denote set of $j$-simplices of $\K^i$. Then, we have that $\K^1_0 = \K^1$ and for $j \geq 1$, $\K^1_j = \emptyset$. This implies that $$H_0(\K^1,\F) = \F^{|\K^1|} = \F^4,$$
and $H_n(\K^1,\F) = 0$ for all $n \geq 1$. 

For $\K^2$, we have $\K^2_0 = \{[1],[2],[3],[4] \} $, $\K^2_1 = \{[1,2],[2,3],[3,4],[1,4] \}$ and for $j \geq 2 $, $\K^2_j = \emptyset $. The chain complex of $\K^2$ looks as $\langle \K^2_1 \rangle_\F \xrightarrow{\partial_1} \langle \K^2_0 \rangle_\F \xrightarrow{\partial_0} 0$. We have that $\B_0(\K^2,\F) = \mathrm{image}(\partial_1) =  \Span \left( [2]-[1], [3]-[2], [4]-[3], [4] - [1] \right)$. Clearly, $\dim(\B_0(\K^2,\F)) = 3$, and thus we obtain that 
$$H_0(\K^2,\F) = \frac{\F^4}{\F^3} = \F.$$ 
We also obtain that $\dim(\mathrm{ker}(\partial_1)) = 1$, and thus $$H_1(\K^2,\F) = \F.$$ 
Clearly, $H_n(\K^2,\F) = 0$ for all $n \geq 2$. 

For $\K^3$, we have $\K^3_0 = \{[1],[2],[3],[4] \} $, $\K^3_1 = \{[1,2],[2,3],[3,4],[1,4],[1,3],[2,4] \}$, $\K^3_2 = \{[1,2,3], [1,2,4], [1,3,4], [2,3,4] \}$, $\K^3_3 = \{[1,2,3,4] \}$, and $\K^3_n = \emptyset$ for all $n \geq 4$. The chain complex of $\K^3$ looks as 
$$\mathcal{C}_\ast(\K^3,\F) = \langle \K^3_3 \rangle_\F \xrightarrow{\partial_3} \langle \K^3_2 \rangle_\F \xrightarrow{\partial_2} \langle \K^3_1 \rangle_\F \xrightarrow{\partial_1} \langle \K^3_0 \rangle_\F \xrightarrow{\partial_0} 0.$$ 
We have $\B_0(\K^3,\F) = \mathrm{image}(\partial_1) = \Span \big([2]-[1], [3]-[2], [4]-[3], [4]-[1], [3]-[1], [4]-[2] \big) $. We observe that $\dim(\B_0(\K^3,\F)) = 3$, and therefore 
$$H_0(\K^3,\F) = \frac{\F^4}{\F^3} = \F.$$
We have $\Z_1(\K^3,\F) = \mathrm{ker}(\partial_1) =\Span \big([1,2]+[2,3]-[1,3], [2,3]+[3,4]-[2,4], [1,3]+[3,4]-[1,4], [1,2]+[2,4]-[1,4], [1,2]+[2,3]+[3,4]-[1,4] \big).$
It is an easy exercise to check that $\dim(\Z_1(\K^3,\F)) = 3$. We have $\B_1(\K^3,\F) = \mathrm{image}(\partial_2) = \Span \big([2,3]-[1,3]+[1,2], [2,4]-[1,4]+[1,2], [3,4]-[1,4]+[1,3], [3,4]-[2,4]+[2,3]\big)$. It is easy to see that $\dim(\B_1(\K^3,\F)) = 3$, and therefore, we obtain
$$H_1(\K^3,\F) = \frac{\F^3}{\F^3} = 0.$$
We have $\Z_2(\K^3,\F) = \mathrm{ker}(\partial_2) = \Span \big([1,2,3] + [1,3,4] - [1,2,4]-[2,3,4] \big)$, and therefore $\dim(\Z_2(\K^3,\F)) = 1$. We have $\B_2(\K^3,\F) = \mathrm{image}(\partial_3) = \Span \big( [2,3,4] - [1,3,4] + [1,2,4] - [1,2,3] \big) = \Z_2(\K^3,\F).$ This implies that 
$$H_2(\K^3,\F) = 0.$$
Clearly, for $n \geq 3$, $H_n(\K^3,\F) = 0$. 

Therefore, the homology groups of $\kvr_\bullet(X)$ are as follows:
$$ H_0(\kvr_\bullet(X),\F) = \F^4 \rightarrow \ldots \rightarrow \F \rightarrow \ldots,  $$
with the transition $\F^4 \rightarrow \F^1$ occurring at $\delta = 1$,
$$ H_1(\kvr_\bullet(X),\F) = 0 \rightarrow \ldots \rightarrow \F \rightarrow \ldots \rightarrow 0 \rightarrow \ldots, $$
with the transition $0 \rightarrow \F$ occurring at $\delta = 1$ and the transition $\F \rightarrow 0$ occurring at $\delta = 2$; and $H_n(\kvr_\bullet(X),\F) = 0$ for all $n \geq 2$. We now calculate $\dgm_0^\vr(X)$ and $\dgm_1^\vr(X)$ using maps $\mathbf{S}$ and $\mathbf{T}$. 

We have that $\mathrm{spec}(X) = \{0,1,2 \}$, and therefore for $k=0,1$, $\mathbf{S}(H_k \circ \kvr_\bullet(X), \mathrm{spec}(X))$ are persistence vector spaces of length $3$ given by
$$ \Vb_0^X = \F^4 \xrightarrow{v_{1,2}} \F \xrightarrow{v_{2,3}} \F~\mbox{and}~\Vb_1^X = 0 \xrightarrow{w_{1,2}} \F \xrightarrow{w_{2,3}} 0.$$
Here, we have that for all permutations $\sigma \in S_4$, $v_{1,2}(\sigma(1_\F,0_\F,0_\F,0_\F)) = 1_\F $, and $v_{2,3}(1_\F) = 1_\F$. The maps of $\Vb_1^X$ are the trivial maps. Clearly, we have
$$ \Vb_0^X \cong \big(\F \rightarrow \F \rightarrow \F \big)~\oplus~\big(\F \rightarrow 0 \rightarrow 0 \big)~\oplus~\big(\F \rightarrow 0 \rightarrow 0 \big)~\oplus~\big(\F \rightarrow 0 \rightarrow 0\big).$$
Thus, we obtain that $\dgm(\Vb_0^X) = \lc (1,3), (1,1), (1,1), (1,1) \rc $ and $\dgm(\Vb_1^X) = \lc (2,2) \rc $. We now apply the map $\mathbf{T}(\cdot, \mathrm{spec}(X))$ to both $\dgm(\Vb_0^X)$ and $\dgm(\Vb_1^X)$ to obtain
$$ \dgm_0^\vr(X) = \mathbf{T}(\dgm(\Vb_0^X), \mathrm{spec}(X)) = \lc (0,\infty), (0,1),(0,1),(0,1) \rc, $$
$$ \dgm_1^\vr(X) = \mathbf{T}(\dgm(\Vb_1^X), \mathrm{spec}(X)) = \lc (1,2) \rc. $$
\end{example}

\subsection{Interpretation of Clustering via 0-Dimensional Persistence Diagram} \label{subsec:interpretation}

Let $(X,d_X) \in \M$. We now make some observations about $\dgm_0^{\vr}(X)$. We first observe that the number of intervals in $\dgm_0^{\vr}(X)$ is equal to $|X|$. This is because by definition, $\K_0^{\vr}(X)$ consists of only singletons, and we know from Lemma \ref{lem:betti0} that $H_0(\K_0^{\vr}(X),\F) = \F^{r}$, where $r$ is the number of connected components of $\K_0^{\vr}(X)$. Thus, 
$$H_0(\K_0^{\vr}(X),\F) = \F^{|X|}.$$ 
This implies that there are $|X|$ intervals in the decomposition of $H_0(\K_\bullet^\vr(X),\F)$ into interval persistence vector spaces. We simultaneously obtain that if $H_0(\K_\bullet^\vr(X),\F) \cong \bigoplus_{i=1}^{|X|} \I(b_i,d_i) $, then $b_i=0$ for all $i \in [1:|X|] $. In the next proposition, we explicitly determine the intervals $\lc (b_i,d_i) \rc_{i=1}^{|X|}$ in $\dgm_0^{\vr}(X)$ and provide a method of associating an interval with every point of $X$.

We now recall the single linkage dendrogram of $(X,d_X)$, denoted by $\theta_X$. We have that for every $t \geq 0$, $\theta_X(t)$ is a partition of $X$, and for $t' \geq t$, $\theta_X(t)$ is finer than $\theta_X(t')$. Let $ \mathfrak{V}_\F : \pc \rightarrow \V_\F$ denote the functor defined as $\mathfrak{V}_\F(\{B_1, \ldots, B_k \}) = \F^k$. Here, $\{B_1, \ldots, B_k \} $ denotes a partition of some finite metric space. It is straightforward to check that $\mathfrak{V}_\F$ is a functor. We observe that $\{\mathfrak{V}_\F \circ \theta_X(t) \rightarrow \mathfrak{V}_\F \circ \theta_X(t') \}_{0 \leq t \leq t'}$ forms a pointwise finite dimensional persistence vector space, and thus admits an interval decomposition. Then, we have the following proposition. 

\begin{proposition}
For all $(X,d_X) \in \M$,
$$ \{\mathfrak{V}_\F \circ \theta_X(t) \rightarrow \mathfrak{V}_\F \circ \theta_X(t') \}_{0 \leq t \leq t'} \cong H_0(\K_\bullet^\vr(X),\F). $$
\end{proposition}

The proof of the above proposition is an easy exercise, and uses the observation that in the persistence vector space $H_0(\K_\bullet^\vr(X),\F)$, every time consecutive vector spaces in the sequence $H_0(\K_\bullet^\vr(X),\F)$ are different, there is a merging of bars in the dendrogram $\theta_X$. Thus, we observe that $H_0(\K_\bullet^\vr(X),\F)$ is equivalent to single linkage clustering of $X$, and therefore persistent homology generalizes clustering. In the next proposition, we explicitly determine the interval decomposition of $H_0(\K_\bullet^\vr(X),\F)$, and provide a map that associates to every element of $X$, an interval of this decomposition.

\begin{proposition}
Let $(X,d_X)$ be a finite metric space, with $|X| = n$. Let $X = \{x_1, x_2,   \ldots, x_n\}$, and let $u_X $ be the maximal sub-dominant ultrametric (Definition \ref{def:ultrametric}) on $X$. Then, 
$$H_0(\kvr(X),\mathbb{Z}_2) \cong \bigoplus_{2 \leq i \leq n} \mathbb{I}(0, \min_{k < i}u_X(x_k, x_i)) \bigoplus \mathbb{I}(0, \infty). $$
\end{proposition}

\begin{proof}
We refer the reader to chapter $7$ of \cite{comp-top} for ideas used in this proof. In chapter $7.1$ of \cite{comp-top}, the authors provide an algorithm for determining the intervals in $\dgm_k^\vr(X)$ for all $k \in \mathbb{Z}_+$. A proof of correctness of this algorithm also appears in \cite[Chapter 7.1]{comp-top}. Here, we briefly describe their algorithm for $k=0$.

We define an arbitrary ordering $x_1 < x_2 <  \ldots < x_n $ of elements of $X$. Let $C_1 \subset 2^X $ denote the collection of subsets of $X$ of cardinality at most $2$. By definition, $C_1$ is a simplicial complex. We fix the following notation here: we have $\{x_i,x_j \} \in C_1 $ only for $i < j$. We define a function $f:C_1 \rightarrow \mathbb{R}_+ $ as $f(\{x_i, x_j \}) = u_X(x_i, x_j) $ and $f(\{x_i \}) = 0$, for all $1 \leq i \leq j \leq n$. We now define an ordering $<_S $ on elements of $C_1$ as follows: we first fix $\{x_1\} <_S \{x_2\} <_S \ldots <_S \{x_n\} $. In order to determine the ordering of subsets of cardinality $2$, we compare their values on the function $f$. We set $\{x_i, x_j \} <_S \{x_k, x_l \} $ if $f(\{x_i, x_j \}) \leq f(\{x_k,x_l \})$. If subsets $\{x_i, x_j \}, \{x_k, x_l \} $ are such that $f(\{x_i, x_j \}) = f(\{x_k, x_l \})$, then $\{x_i,x_j \} <_S \{x_k,x_l \} $ if and only if $i < k $ or $i = k, j < l$. Thus, we use lexicographic ordering on elements of $C_1 $ with same value on the function $f$. We have that every element of $C_1 $ of cardinality $2$ is a $1$-dimensional face of itself, and every element of cardinality $1$ is a $0$-dimensional face of any set of cardinality $2$ containing it. The ordering $<_S $ satisfies that if $f(\{x_i, x_j \}) < f(\{x_k,x_l \}) $, then $\{x_i,x_j \} <_S \{x_k,x_l \} $, and the faces $\{x_i \},\{x_j \} $ of $\{x_i,x_j \}$ satisfy $\{x_i \}, \{x_j \} <_S \{x_i,x_j \} $. Thus, we have a \emph{compatible ordering} of the faces of $C_1$. 

We now write the \emph{boundary matrix} $B$ using this ordering of faces. The boundary matrix is a binary square matrix of size $n + \frac{n(n-1)}{2} $. The size of $B$ is equal to $|C_1|$, and the rows and columns of $B$ correspond to elements of $C_1$ ordered according to relation $<_S$. By abuse of notation, we name the rows and columns of $B$ on their corresponding elements in $C_1$. The columns of $B$ corresponding to the singleton sets of $C_1$ are set to be zero, while for all $1 \leq i \leq j \leq n $, the column $\{x_i, x_j \} $ has $1$ in the rows $\{x_i\} $ and $\{x_j\}$, and zero everywhere else. For every column $\{x_i,x_j \}$ of $B$, we denote by $\low(\{x_i,x_j \}) $, the row of $B$ in which the lowest $1$ of the column $\{x_i,x_j \}$ appears.

We perform some column additions in the matrix $B$ such that in the new matrix, no two columns have their lowest $1$ in the same row. This is done as follows: for simplicity, we number the columns from $1$ to $|C_1|$, with the leftmost column being numbered $1$ and the rightmost column being numbered $|C_1|$. We scan the boundary matrix from left to right and suppose that $t$ is the first column for which there is a column $s$, $s < t$, satisfying $\low(s) = \low(t)$. In this case, we add column $s$ to column $t$. Now, $\low(s) \neq \low(t)$, but there might be some column $r$, $r < t$ such that $\low(r) = \low(t)$. We then add column $r$ to column $t$. We keep performing such column additions till there is no column to the left of column $t$ with $\low$ value equal to $\low(t)$. We then proceed to column $t+1$ and repeat. In the end, we obtain a matrix in which no two columns have their lowest $1$ in the same row. This matrix is called the \emph{reduced matrix} and is denoted by $R$. Now, for every non-zero column $\{x_i,x_j \} $ in $R$ having lowest $1$ in the row $\{x_j \} $, the interval $(f(\{x_i\}), f(\{x_i, x_l \})) $ belongs to $\dgm_0^{\vr}(X)$. In addition, the interval $[0, \infty)$ belongs to $\dgm_0^\vr(X)$. This concludes the algorithm used to determine intervals in $\dgm_0^\vr(X)$. 

We now provide an example in order to illustrate the above algorithm. Let $X = \{x_1,x_2,x_3 \} $, with $d_X(x_1,x_2) =1, d_X(x_2,x_3)=2$ and $d_X(x_3,x_1)=3$. We have 
$$C_1 = \{\{x_1\}, \{x_2 \}, \{x_3 \}, \{x_1,x_2 \}, \{x_2,x_3 \}, \{x_1,x_3 \} \}.$$ 
The function $f:C_1 \rightarrow \mathbb{R}_+$ is defined as follows: $f(\{x_1\}) = f(\{x_2 \}) = f(\{x_3 \}) = 0, f(\{x_1,x_2 \}) = u_X(x_1,x_2)=1, f(\{x_2,x_3 \}) = u_X(x_2,x_3) = 2, f(\{x_1,x_3 \}) = u_X(x_1,x_3)=2$. Thus, we have $\{x_1 \} <_S \{x_2 \} <_S \{x_3 \} <_S \{x_1,x_2 \} <_S \{x_1,x_3 \} <_S \{x_2,x_3 \} $. The boundary matrix is the following:
$$ B = \bordermatrix{& \{x_1 \} & \{x_2 \} & \{x_3 \} & \{x_1,x_2 \} & \{x_1,x_3 \} & \{x_2,x_3 \} \cr 
\{x_1 \} & 0 & 0 & 0 & 1 & 1 & 0 \cr 
\{x_2 \} & 0 & 0 & 0 & 1 & 0 & 1 \cr 
\{x_3 \} & 0 & 0 & 0 & 0 & 1 & 1\cr 
\{x_1,x_2 \} & 0 & 0 & 0 & 0 & 0 & 0 \cr 
\{x_1,x_3\} & 0 & 0 & 0 & 0 & 0 & 0 \cr 
\{x_2,x_3 \} & 0 & 0 & 0 & 0 & 0 & 0 } $$
The reduced matrix obtained after performing the required column operations is the following:
$$ R = \bordermatrix{& \{x_1 \} & \{x_2 \} & \{x_3 \} & \{x_1,x_2 \} & \{x_1,x_3 \} & \{x_2,x_3 \} \cr 
\{x_1 \} & 0 & 0 & 0 & 1 & 1 & 0 \cr 
\{x_2 \} & 0 & 0 & 0 & 1 & 0 & 0 \cr 
\{x_3 \} & 0 & 0 & 0 & 0 & 1 & 0\cr 
\{x_1,x_2 \} & 0 & 0 & 0 & 0 & 0 & 0 \cr 
\{x_1,x_3\} & 0 & 0 & 0 & 0 & 0 & 0 \cr 
\{x_2,x_3 \} & 0 & 0 & 0 & 0 & 0 & 0 } $$
The algorithm now implies that $\dgm_0^\vr(X) = \lc (0,1),(0,2),(0,\infty) \rc $. We now use the following claim to determine $\dgm_0^\vr(X)$ for any $(X,d_X) \in \M$.

\begin{claim}
For every $\{x_k \} \in C_1 $, the unique column in the reduced matrix $R$ with lowest $1$ in the row $\{x_k \}$ is the leftmost column in the boundary matrix $B$ with lowest $1$ in the row $\{x_k \}$.
\end{claim}

\begin{proof}[Proof of Claim]
Consider the row $\{x_k \}$ in the matrix $B$, and the column in which $1$ appears for the first time in this row. If this $1$ is the lowest element of its column, then the column is $\{x_i,x_k \} $ for some $i < k$, and the interval $(f(\{x_k \}, f(\{x_i,x_k \})$ is added to $\dgm_0^{\vr}(X)$. We observe that $(f(\{x_k \}, f(\{x_i,x_k \}) = (0, \min_{i < k}u_X(x_i,x_k))$. Now, suppose that the first $1$ of row $\{x_k \}$ is not the lowest element of its column. Thus, such a column is $\{x_k,x_m \} $ for $k < m$. In the column additions performed to obtain the reduced matrix, this particular $1$ becomes the lowest $1$ of some other column in two ways. First, if some column $\{x_i,x_m \}, i < k $, on the left of column $\{x_k,x_m \} $ is added to the column $\{x_k,x_m \} $, and second if column $\{x_k,x_m \} $ is added to some column $\{x_i,x_m \}, i<k$ on its right. We first consider the case where there is a column $\{x_i,x_m \}, i<m $ to the left of column $\{x_k,x_m \} $. In this case, we have $ u_X(x_i,x_m) \leq u_X(x_k,x_m)$. Now, $u_X(x_i,x_k) \leq \max \{u_X(x_i, x_m), u_X(x_k, x_m) \} = u_X(x_k, x_m) $. Since $i < k$, we have that $\{x_i,x_k \} <_S \{x_k,x_m \} $. This contradicts the assumption that the first $1$ in the row $\{x_k \} $ appears in the column $\{x_k,x_m \} $. Therefore, there is no column $\{x_i, x_m \}$ with $i < m$ to the left of $\{x_k,x_m \} $. 

We now consider the second case i.e.~there is a column $\{x_i,x_m \}, i < k $ to the right of $\{x_k,x_m \}$. Here, we have $u_X(x_k,x_m) \leq u_X(x_i,x_m)$. Since 
$$u_X(x_i,x_k) \leq \max \{u_X(x_k,x_m), u_X(x_i, x_m) \} = u_X(x_i,x_m),$$
we have that $\{x_i,x_k \} <_S \{x_i,x_m \} $. Thus, we have a column $\{x_i,x_k \} $ whose lowest $1$ is in the row $x_k$, and this column appears before the column $\{x_i,x_m \}$. Thus, if column $\{x_k,x_m \}$ is added to column $\{x_i, x_m \}$, the lowest $1$ of $\{x_i,x_m \} $ is in row $\{x_k \}$, but this does not affect the column $\{x_i,x_k\} $. This implies that the column $\{x_k,x_m \} $ has no affect on the leftmost column of the boundary matrix with lowest $1$ in the row $\{x_k \}$. Thus, we obtain that the interval corresponding to the row $\{x_k \} $ comes from the first column whose lowest $1$ is in row $\{x_k \}$. This proves the claim.
\end{proof}

Note that this proof also suggests a method of associating an interval to every element of $X$. In particular, the interval associated with $x_k \in X$ is the interval associated with row $\{ x_k \}$ and the first column of the boundary matrix with lowest $1$ in row $\{x_k \}$. By definition, the value of such a column under function $f$ is $\min_{i<k}u_X(x_i,x_k)$. Thus, we have that the element $x_k \in X, k > 1$ is associated with the interval $(0, \min_{i < k}u_X(x_i,x_k) ) $. We associate the interval $[0, \infty) $ with $x_1$.
\end{proof}

Now, suppose that given finite metric spaces $(X,d_X)$ and $(Y,d_Y)$, we construct the Vietoris-Rips simplicial complexes $\K_\bullet^{\vr}(X)$ and $\K_\bullet^{\vr}(Y)$. Then, for a fixed $k \in \mathbb{Z}_+$, we compute the persistence vector spaces $H_k \circ \K_\bullet^{\vr}(X)$ and $H_k \circ \K_\bullet^{\vr}(Y)$, as well as their respective persistence diagrams. Suppose that we have a method of comparing two metric spaces as well as two persistence diagrams. Then, a natural question to ask is, if $(X,d_X)$ and $(Y,d_Y)$ are ``almost identical", then how do the persistence diagrams associated to $H_k \circ \K_\bullet^{\vr}(X)$ and $H_k \circ \K_\bullet^{\vr}(Y)$ compare. The next section focuses towards formalizing this question and then answering it.

\section{Stability of Invariants} \label{sec:stability}

In this section, we formalize the following question: if $(X,d_X), (Y,d_Y) \in \M $ are almost identical, then how do their respective $k$-persistence diagrams, $\dgm_k^{\vr}(X)$ and $\dgm_k^{\vr}(Y) $ compare. This is done by defining a notion of dissimilarity between metric spaces, as well as between persistence diagrams. Therefore, we now define a notion of distance between metric spaces, and a notion of distance between persistence vector spaces as well as between persistence diagrams. 
\subsection{Gromov-Hausdorff Distance}

In thie section, we define a notion of distance between two finite metric spaces. Let $(X, d_X), (Y, d_Y) \in \mathcal{M} $. We say that $(X,d_X)$ and $(Y,d_Y)$ are identical if they are \emph{isometric}.

\begin{definition}[Isometry]
An isometry between $(X,d_X), (Y,d_Y) \in \M$ is a map $\phi:X \rightarrow Y$ such that $\phi$ is surjective, and for all $x,x' \in X$, $d_X(x,x') = d_Y(\phi(x), \phi(x')) $.
\end{definition}
Note that the condition $d_X(x,x') = d_Y(\phi(x), \phi(x')) $ ensures that $\phi$ is injective. Thus, if $(X,d_X), (Y,d_Y) \in \M$ are isometric, then there exists a bijective and distance preserving map $\phi:X \rightarrow Y$. Since $\phi$ is a bijection between finite metric spaces, it has an inverse, say $\psi:Y \rightarrow X$, with $\phi \circ \psi = \mathrm{id}_Y $ and $\psi \circ \phi = \mathrm{id}_X$. We now define the \emph{distortion} and \emph{co-distortion} of maps $\phi$ and $\psi$.

\begin{definition}[Distortion]
Given $(X,d_X), (Y,d_Y) \in \M $, the distortion of a map $f:X \rightarrow Y $ is defined as
$$ \dis(f) := \max_{x,x' \in X}|d_X(x,x') - d_Y(f(x), f(x'))|. $$
\end{definition}

\begin{definition}[Co-distortion]
Given $(X,d_X), (Y,d_Y) \in \M $, and maps $f:X \rightarrow Y $, $g:Y \rightarrow X$, the co-distortion of $f$ and $g$ is defined as
$$ C(f,g) := \max_{x \in X, y \in Y}|d_X(x,g(y)) - d_Y(y,f(x)) |. $$
\end{definition}

We observe that if $(X,d_X), (Y,d_Y) \in \M$ are isometric, then there exist maps $\phi:X \rightarrow Y$ and $\psi:Y \rightarrow X$ such that $\dis(\phi) = 0,~\dis(\psi) = 0$ and $C(\phi,\psi) = 0 $. We now want to relax the notion of isometry between metric spaces to the notion of $\eta$-isometry, for some $\eta > 0$.

\begin{definition}[$\eta$-isometry]
Given $(X,d_X), (Y,d_Y) \in \M $ and $\eta > 0$, maps $\phi:X \rightarrow Y $ and $\psi:Y \rightarrow X$ constitute an $\eta$-isometry between $(X,d_X)$ and $(Y,d_Y)$ if $\dis(\phi) \leq \eta $, $\dis(\psi) \leq \eta$ and $C(\phi,\psi) \leq \eta$.
\end{definition}

We observe that if $\phi:X \rightarrow Y$ and $\psi:Y \rightarrow X$ constitute an $\eta$-isometry between $(X,d_X)$ and $(Y,d_Y)$, then $\phi \circ \psi $ is an ``approximate identity" on $Y$ and $\psi \circ \phi$ is an ``approximate identity" on $X$. This means the following: for every $x \in X$, we have
$$ d_X(x,\psi \circ \phi(x)) = |d_X(x,\psi \circ \phi(x) ) - d_Y(\phi(x),\phi(x)) | \leq C(\phi,\psi) \leq \eta. $$
Similarly, for every $y \in Y$, we have $d_Y(y, \phi \circ \psi(y)) \leq \eta$. Now, we are ready to define a notion of distance between metric spaces.

\begin{definition}[Gromov-Hausdorff distance]
The Gromov-Hausdorff distance between $(X, d_X),~(Y, d_Y) \in \M$ is defined as
$$ \dgh((X, d_X), (Y, d_Y)) := \frac{1}{2} \cdot \min \{\eta \geq 0~|~\exists~\mbox{an}~\eta\mbox{-isometry~between}~X~\mbox{and}~Y \}. $$ 
\end{definition}

\begin{theorem}[{\cite[Theorem 7.3.30]{burago-book}}]
The function $\dgh : \mathcal{M} \times \mathcal{M} \rightarrow \mathbb{R}_+ $ is non-negative, symmetric and satisfies the triangle inequality; moreover $\dgh((X,d_X),(Y,d_Y)) = 0$ if and only if $X$ and $Y$ are isometric.
\end{theorem}

Thus, we now have a notion of distance between finite metric spaces. The next step is to define a notion of distance between persistence vector spaces.

\subsection{Interleaving Distance}

We recall that $P\V(\F)$ denotes the category of all persistence vector spaces over field $\F$. For $\Vb,\Wb \in P\V(\F)$, $\Vb = \{V_{\delta} \xrightarrow{v_{\delta,\delta'}} V_{\delta'} \}_{\delta \leq \delta'} $ and $\Wb = \{W_{\delta} \xrightarrow{w_{\delta,\delta'}} W_{\delta'} \}_{\delta \leq \delta'} $, we recall that an isomorphism $\alpha:\Vb \rightarrow \Wb $ consists of maps $\alpha_\delta:V_\delta \rightarrow W_\delta$ for all $\delta$, such that the following diagram commutes for all $\delta \leq \delta'$,
\[
\begin{tikzcd}
V_\delta \arrow[r, "v_{\delta,\delta'}"] \arrow[d, "\alpha_\delta"] 
& V_{\delta'} \arrow[d, "\alpha_{\delta'}"] \\
W_\delta \arrow[r, "w_{\delta, \delta'}"] 
& W_{\delta'}
\end{tikzcd}
\]
and $\alpha_\delta: V_\delta \rightarrow W_\delta$ is an isomorphism of vector spaces for all $\delta \in \mathbb{R}_+$. We now relax this notion of isomorphism between persistence vector spaces.
\begin{definition}[$\eta$-interleaving]
Let $\eta \geq 0 $ be fixed. Given a field $\F$, an $\eta$-interleaving between $\Vb,\Wb \in P\V(\F)$, $\Vb = \{V_{\delta} \xrightarrow{v_{\delta,\delta'}} V_{\delta'} \}_{\delta \leq \delta'} $ and $\Wb = \{W_{\delta} \xrightarrow{w_{\delta,\delta'}} W_{\delta'} \}_{\delta \leq \delta'} $ consists of maps
$$ \alpha_\delta:V_\delta \rightarrow W_{\delta + \eta},~~ \beta_\delta:W_\delta \rightarrow V_{\delta + \eta}~\forall~\delta  $$
such that the following diagrams commute for all $\delta$:
\[
\begin{tikzcd}
V_\delta \arrow[rr,"v_{\delta, \delta+2 \eta}"] \arrow[dr, "\alpha_\delta"] & & V_{\delta + 2 \eta} \\
& W_{\delta+\eta} \arrow[ur,"\beta_{\delta + \eta}"] &
\end{tikzcd}
\begin{tikzcd}
& V_{\delta + \eta} \arrow[dr, "\alpha_{\delta+\eta}"] & \\
W_\delta \arrow[rr, "w_{\delta, \delta + 2\eta}"] \arrow[ur,"\beta_\delta"] & & W_{\delta+2 \eta} 
\end{tikzcd}
\]
The above conditions are referred to as the triangle conditions. We also want the following diagrams to commute for all $\delta' \geq \delta$. These conditions are referred to as the parallelogram conditions. 
\[
\begin{tikzcd}
V_\delta \arrow[rr, "v_{\delta, \delta'}"] \arrow[dr, "\alpha_\delta"] &  & V_{\delta'} \arrow[dr, "\alpha_{\delta'}"] & \\
& W_{\delta + \eta} \arrow[rr, "w_{\delta+\eta, \delta' + \eta}"] &  & W_{\delta'+ \eta } 
\end{tikzcd}
\begin{tikzcd}
& V_{\delta+\eta} \arrow[rr, "v_{\delta+\eta, \delta' + \eta} "] & & V_{\delta' + \eta} \\
W_{\delta} \arrow[rr, "w_{\delta, \delta'} "] \arrow[ur, "\beta_\delta "] & & W_{\delta'} \arrow[ur, "\beta_{\delta'} "] & 
\end{tikzcd}
\]
We set $\alpha = \{\alpha_\delta \}_\delta$ and $\beta = \{\beta_\delta \}_\delta$, and say that $(\alpha, \beta)$ is an $\eta$-interleaving between $\Vb$ and $\Wb$. 
\end{definition}
We observe that an $\eta$-interleaving is indeed a generalization of isomorphism between persistence vector spaces. In fact, if $(\alpha,\beta)$ is a $0$-interleaving between $\Vb$ and $\Wb$, then it is straightforward to see that for every $\delta$, $\beta_\delta \circ \alpha_\delta = \mathrm{id}_{V_\delta} $ and $\alpha_\delta \circ \beta_\delta = \mathrm{id}_{W_\delta}$. Thus, $\alpha_\delta $ and $\beta_\delta$ become isomorphisms of vector spaces for all $\delta$. We are now ready to define the interleaving distance between persistence vector spaces.

\begin{definition}[Interleaving distance]
Given a field $\F$, the interleaving distance between $\Vb, \Wb \in P\V(\F)$ is defined as
$$ \di(\Vb,\Wb) := \inf \{\eta \geq 0~|~\Vb~\mbox{and}~\Wb~\mbox{are}~\eta\mbox{-interleaved}\}. $$
\end{definition}

\begin{proposition}[\cite{chazal16}]
The function $\di:P\V(\F) \times P\V(\F) \rightarrow \mathbb{R}_+ \cup \infty$ is non-negative, symmetric and satisfies the triangle inequality. However, $\di(\Vb,\Wb)$ may take value $\infty$ and $\di(\Vb,\Wb) $ might be zero even if $\Vb$ and $\Wb$ are not isomorphic. 
\end{proposition}

We are now ready to prove the following stability theorem. For $(X,d_X) \in \M$, we recall that $\K_\bullet^{\vr}(X) $ is the Vietoris-Rips filtered simplicial complex associated with $(X,d_X)$.

\begin{theorem}[Stability of Vietoris-Rips persistent homology] \label{thm:stability}
For all $(X,d_X), (Y,d_Y) \in \M$ and $k \in \mathbb{N}$,
$$ 2 \cdot \dgh((X, d_X), (Y, d_Y)) \geq \di(H_k \circ \K_\bullet^\vr(X), H_k \circ \K^\vr_\bullet(Y)). $$
\end{theorem}

\begin{proof}
Let $\eta \geq 0$ be fixed. We show that if $X$ and $Y$ are $\eta$-isometric, then $H_k \circ \K^\vr_\bullet(X)$ and $H_k \circ \K^\vr_\bullet(Y)$ are $\eta$-interleaved. Suppose that $(X,d_X)$ and $(Y,d_Y)$ are $\eta$-isometric. Then, there exist maps $\phi:X \rightarrow Y$ and $\psi:Y \rightarrow X$ such that $\dis(\phi) \leq \eta$, $\dis(\psi) \leq \eta$ and $C(\phi,\psi) \leq \eta$. For some $\delta \geq 0$, let $\sigma \in \K_\delta^{\vr}(X)$. This implies that for all $x,x' \in \sigma$, $d_X(x,x') \leq \delta$. Since $\dis(\phi) \leq \eta$, we obtain that $d_Y(\phi(x), \phi(x')) \leq \delta + \eta$ for all $x,x' \in \sigma$. Thus, $\phi(\sigma) \in \K_{\delta + \eta}^{\vr}(Y) $. Thus, for every $\delta > 0$, $\phi$ induces a map $$\phi_\delta : \K_\delta^{\vr}(X) \rightarrow \K_{\delta + \eta}^{\vr}(Y).$$ Similarly, for every $\delta > 0$, we obtain maps 
$$\psi_\delta: \K_{\delta}^{\vr}(Y) \rightarrow \K_{\delta + \eta}^{\vr}(X).$$ Thus, we obtain the following diagrams:
\[
\begin{tikzcd}
\K_{\delta }^{\vr}(X) \arrow[rr, hook, "i^X_{\delta, \delta+2 \eta} "] \arrow[dr, "\phi_\delta " ] & & \K_{\delta + 2\eta}^{\vr}(X) \\
& \K_{\delta + \eta}^{\vr}(Y) \arrow[ur,"\psi_{\delta + \eta}"] &
\end{tikzcd}
\begin{tikzcd}
& \K_{\delta + \eta}^{\vr}(X) \arrow[dr,"\phi_{\delta + \eta}"] & \\
\K_{\delta }^{\vr}(Y) \arrow[rr, hook, "i^Y_{\delta, \delta+2 \eta} "] \arrow[ur, "\psi_\delta " ] & & \K_{\delta + 2\eta}^{\vr}(Y)
\end{tikzcd}
\]

\[
\begin{tikzcd}[column sep  = 2]
\K_{\delta }^{\vr}(X) \arrow[rr, hook, "i^X_{\delta, \delta'} "] \arrow[dr, "\phi_\delta " ] & & \K_{\delta'}^{\vr}(X) \arrow[dr, "\phi_{\delta'} "] & \\
& \K_{\delta + \eta }^{\vr}(Y) \arrow[rr, hook, "i^Y_{\delta+\eta, \delta'+\eta} "] & & \K_{\delta' + \eta}^{\vr}(Y)
\end{tikzcd}
\begin{tikzcd}[column sep = 2]
& \K_{\delta + \eta }^{\vr}(X) \arrow[rr, hook, "i^X_{\delta+\eta, \delta'+\eta} "] & & \K_{\delta' + \eta}^{\vr}(X) \\
\K_{\delta }^{\vr}(Y) \arrow[rr, hook, "i^Y_{\delta,\delta'} "] \arrow[ur, "\psi_\delta " ] & & \K_{\delta'}^{\vr}(Y) \arrow[ur, "\psi_{\delta'} "] &
\end{tikzcd}
\]
If the above four diagrams were to commute, then, since $H_k(\ast,\F):\Sc \rightarrow \V(\F)$ is a functor for every $k \in \mathbb{N}$, the following diagrams will also commute: fix $k \in \mathbb{N}$, and let $V^X_\delta := H_k(\K_\delta^{\vr}(X),\F)$, $V^Y_{\delta} := H_k(\K_{\delta}^{\vr}(Y),\F) $, $\phi_\delta^\ast := H_k(\phi_\delta)$ and $\psi_\delta^\ast := H_k(\psi_\delta) $.
\[
\begin{tikzcd}
V_{\delta }^X \arrow[rr] \arrow[dr, "\phi^\ast_\delta " ] & & V_{\delta + 2\eta}^X \\
& V_{\delta + \eta}^Y \arrow[ur,"\psi^\ast_{\delta + \eta}"] &
\end{tikzcd}
\begin{tikzcd}
& V_{\delta + \eta}^X \arrow[dr,"\phi^\ast_{\delta + \eta}"] & \\
V_{\delta }^Y \arrow[rr] \arrow[ur, "\psi^\ast_\delta " ] & & V_{\delta + 2\eta}^Y
\end{tikzcd}
\]

\[
\begin{tikzcd}
V_{\delta }^X \arrow[rr] \arrow[dr, "\phi^\ast_\delta " ] & & V_{\delta'}^X \arrow[dr, "\phi^\ast_{\delta'} "] & \\
& V_{\delta + \eta }^Y \arrow[rr] & & V_{\delta' + \eta}^Y
\end{tikzcd}
\begin{tikzcd}
& V_{\delta + \eta }^X \arrow[rr] & & V_{\delta' + \eta}^X \\
V_{\delta }^Y \arrow[rr] \arrow[ur, "\psi^\ast_\delta " ] & & V_{\delta'}^Y \arrow[ur, "\psi^\ast_{\delta'} "] &
\end{tikzcd}
\]
We recall that given simplicial complexes $\K$ and $\Lm$, simplicial maps $f,g:\K \rightarrow \Lm$ are called contiguous if for every simplex $\sigma \in \K$, $f(\sigma) \cup g(\sigma)$ is a simplex in $\Lm$. We have from Lemma \ref{lem:contig} that for such maps, $H_k(f) = H_k(g):H_k(\K,\F) \rightarrow H_k(\Lm,\F)$ for all $k \in \mathbb{N}$. Therefore, it suffices to show that the following maps in the four diagrams on $\kvr(X)$ and $\kvr(Y)$ are contiguous \cite{munkresbook}:
\begin{enumerate}
\item $i^X_{\delta, \delta+2 \eta},~\psi_{\delta+\eta} \circ \phi_\delta : \K_\delta^{\vr}(X) \rightarrow \K_{\delta+2\eta}^{\vr}(X). $
\item $i^Y_{\delta, \delta+2 \eta},~\phi_{\delta+\eta} \circ \psi_\delta: \K_{\delta}^{\vr}(Y) \rightarrow \K_{\delta+2 \eta}^{\vr}(Y). $
\item $\phi_{\delta'} \circ i^X_{\delta,\delta'},~i^Y_{\delta+\eta, \delta'+\eta} \circ \phi_\delta:\K_\delta^{\vr}(X) \rightarrow \K_{\delta'+\eta}^{\vr}(Y). $
\item $\psi_{\delta'} \circ i^Y_{\delta, \delta'},~i^X_{\delta+\eta, \delta'+\eta} \circ \psi_\delta:\K_\delta^{\vr}(Y) \rightarrow \K_{\delta'+\eta}^{\vr}(X). $
\end{enumerate}
We first consider the pair of maps $i^X_{\delta, \delta+2 \eta},~\psi_{\delta+\eta} \circ \phi_\delta : \K_\delta^{\vr}(X) \rightarrow \K_{\delta+2\eta}^{\vr}(X)$. Let $\sigma \in \K_\delta^{\vr}(X)$ be a simplex. In order to show that $i^X_{\delta, \delta+2 \eta}(\sigma) \cup \psi_{\delta+\eta} \circ \phi_\delta(\sigma)$ is a simplex in $\K_{\delta+2\eta}^{\vr}(X)$, we need to show that $\diam(i^X_{\delta, \delta+2 \eta}(\sigma) \cup \psi_{\delta+\eta} \circ \phi_\delta(\sigma) ) \leq \delta + 2 \eta $. Let $x,x' \in i^X_{\delta, \delta+2 \eta}(\sigma) \cup \psi_{\delta+\eta} \circ \phi_\delta(\sigma) $. If both $x,x' \in  i^X_{\delta, \delta+2 \eta}(\sigma) $, then $d_X(x,x') \leq \delta \leq \delta+2 \eta $. If both $x,x' \in \psi_{\delta+\eta} \circ \phi_\delta(\sigma) $, then there exist $a,a' \in \sigma$ such that $x = \psi_{\delta+\eta} \circ \phi_\delta(a) $ and $x' = \psi_{\delta+\eta} \circ \phi_\delta(a')$. Then, 
$$d_X(x,x') = d_X(\psi_{\delta+\eta} \circ \phi_\delta(a), \psi_{\delta+\eta} \circ \phi_\delta(a')) \leq d_Y(\phi_\delta(a), \phi_\delta(a')) + \eta \leq d_X(a,a') + 2 \eta \leq \delta + 2\eta. $$
Here, the first and second inequalities hold because $\dis(\phi),~\dis(\psi) \leq \eta$, while the third inequality hold because $a,a' \in \sigma \in \K_\delta^{\vr}(X)$. Now, suppose that $x \in i^X_{\delta, \delta+2 \eta}(\sigma)$ and $x' \in \psi_{\delta+\eta} \circ \phi_\delta(\sigma) $. This implies that $x \in \sigma$, and there exists $a' \in \sigma$ such that $x' = \psi_{\delta+\eta} \circ \phi_\delta(a')$. We have 
$$d_X(x,x') = d_X(x, \psi_{\delta+\eta} \circ \phi_\delta(a')) \leq d_Y(\phi_{\delta}(x), \phi_\delta(a')) + \eta \leq d_X(x,a') + 2\eta \leq \delta + 2\eta. $$
Here, the first inequality holds because $C(\phi, \psi) \leq \eta $, the second inequality holds because $\dis(\phi) \leq \eta$ and the third inequality holds because $x,a' \in \sigma$. Thus, we have shown that $\diam(i^X_{\delta, \delta+2 \eta}(\sigma) \cup \psi_{\delta+\eta} \circ \phi_\delta(\sigma) ) \leq \delta + 2 \eta.$ This implies that the maps $i^X_{\delta, \delta+2 \eta},~\psi_{\delta+\eta} \circ \phi_\delta : \K_\delta^{\vr}(X) \rightarrow \K_{\delta+2\eta}^{\vr}(X)$ are contiguous. Thus, we have that for every $k \in \mathbb{N}$, 
$$H_k(i^X_{\delta, \delta+2 \eta}) = H_k(\psi_{\delta+\eta} \circ \phi_\delta) = H_k(\psi_{\delta+\eta}) \circ H_k(\phi_\delta).$$
The last equality holds because $H_k(\ast,\F):\Sc \rightarrow \V(\F)$ is a functor for every $k \in \mathbb{N}$. 

We can similarly show that the remaining three pairs of maps are also contiguous. Thus, we obtain that the four diagrams on $V_\delta^X$ and $V_\delta^Y$ also commute. Thus, we have shown that the persistence vector spaces $H_k \circ \K^\vr_\bullet(X) $ and $H_k \circ \K^\vr_\bullet(Y)$ are $\eta$-interleaved. 

Now, let $\dgh((X,d_X),(Y,d_Y)) \leq \eta$. This implies that there exists a $2 \eta$-isometry between $(X,d_X)$ and $(Y,d_Y)$. By the above arguments, we obtain that $H_k \circ \K^\vr_\bullet(X) $ and $H_k \circ \K^\vr_\bullet(Y)$ are $2 \eta$-interleaved for all $k \in \mathbb{N}$. Thus, $\di(H_k \circ \K^\vr_\bullet(X), H_k \circ \K^\vr_\bullet(Y)) \leq 2 \eta $. This implies that $2 \cdot \dgh((X,d_X),(Y,d_Y)) \geq \di(H_k \circ \K^\vr_\bullet(X), H_k \circ \K^\vr_\bullet(Y))$, and proves the theorem.
\end{proof}

We recall that for $(X,d_X) \in \M$ and $\delta \geq 0$, the \v{C}ech complex is defined as 
$$\check{\mathrm{C}}_\delta(X) = \{ \sigma \subseteq X~| \min_{x \in X} \max_{p \in \sigma} d_X(x,p) \leq \delta \}.$$ 
Let $\check{\mathrm{C}}_\bullet(X) = \{\check{\mathrm{C}}_\delta(X) \subseteq \check{\mathrm{C}}_{\delta'}(X) \}_{\delta \leq \delta'} $. Then, we have the following theorem.

\begin{theorem}[Stability of \v{C}ech complex]
For all $(X,d_X), (Y,d_Y) \in \M$, $k \in \mathbb{N}$ and $\eta \geq 0$, if $X$ and $Y$ are $\eta$-isometric, then the persistence vector spaces $H_k \circ \check{\mathrm{C}}_\bullet(X)$ and $H_k \circ \check{\mathrm{C}}_\bullet(Y)$ are $\eta$-interleaved. 
\end{theorem}

The proof of the above theorem is similar to that of Theorem \ref{thm:stability}. The next step is to define a notion of distance between persistence diagrams called the bottleneck distance.

\subsection{Bottleneck Distance}

We recall that given $n \in \mathbb{Z}_+$ and a field $\F$, $P\V_n(\F)$ denotes the category of pfd persistence vector spaces of length $n$, and $\D$ denotes the collection of all persistence diagrams $\dgm(\Vb)$, $\Vb \in P\V_n(\F)$. In the last section, we showed that for every $\Vb \in P\V_n(\F)$, there exists a multiset of intervals $\lc (b_i,d_i) \rc_{i \in I} $ such that $\Vb \cong \bigoplus_{i \in I}\I(b_i, d_i) $. A trivial fact is the following.

\begin{fact}
Given a multiset of intervals $\lc (b_i,d_i) \rc_{i \in I} $, we can construct a persistence vector space $\Vb$ such that $\Vb \cong \bigoplus_{i \in I}\I(b_i, d_i) $.
\end{fact}

We now define the bottleneck distance on $\D$. We recall that every element of $\D$ is a collection of finite multisets of points $(b,d)$, where $0 \leq b \leq d \leq \infty$.

\begin{definition}[Bottleneck distance]
Let $D = \lc(b_\alpha, d_\alpha)~|~\alpha \in A \rc $ and $D' = \lc(b'_\beta, d'_\beta)~|~\beta \in B \rc $ be elements of $\D$. A partial matching $m:A \rightarrow B$ is a bijection between a subset of $A$ and a subset of $B$, which are then the domain and co-domain of $m$ respectively. Let $M(A,B)$ denote the set of all partial matchings between $A$ and $B$. Given $m \in M(A,B)$, the cost of $m$ is defined as follows:
$$ \cost(m) := \max \left\lbrace \max_{\alpha \in A \setminus \dom(m)} \frac{(d_\alpha - b_\alpha)}{2}, \max_{\beta \in B \setminus \codom(m)} \frac{(d'_\beta - b'_\beta)}{2}, \max_{\alpha \in \dom(m)} \max \{|b_\alpha - b'_{m(\alpha)}|, |d_{\alpha} - d'_{m(\alpha)}| \}  \right\rbrace. $$
The bottleneck distance between $D$ and $D'$ is then defined as
$$ d_{\B}(D,D') := \min_{m \in M(A,B)} \cost(m). $$
\end{definition}

The definition implies that the bottleneck distance is symmetric, non-negative and vanishes if $D=D'$. The next theorem shows that the bottleneck distance satisfies the triangle inequality.
\begin{theorem}[\cite{chazal16}]
For any $D,D',D'' \in \D$, we have
$$ \db(D,D'') \leq \db(D,D') + \db(D',D''). $$
\end{theorem}

We now have the following theorem.

\begin{theorem}[Isometry Theorem\cite{Lesnick2015}] 
Given $n \in \mathbb{Z}_+$ and $\Vb, \Wb \in \V_n(\F)$, we have
$$ \di(\Vb, \Wb) = d_{\B}(\dgm(\Vb), \dgm(\Wb)). $$
\end{theorem}

A direct corollary of the above theorem is the following.

\begin{corollary}
For all $(X,d_X), (Y,d_Y) \in \M$ and $k \in \mathbb{N}$,
$$2 \, \dgh(X,Y) \geq \di(H_k \circ \K^\vr_\bullet(X), H_k \circ \K^\vr_\bullet(Y)) = \db(\dgm_k^{\vr}(X), \dgm_k^{\vr}(Y)). $$
\end{corollary}

It is known that while computing $\dgh(X,Y)$ is NP-hard, there is a polynomial time algorithm \cite{comp-top} for computing $\db(\dgm_k^{\vr}(X), \dgm_k^{\vr}(Y))$ for all $k \in \mathbb{N}$.

Furthermore, the inequality $2 \, \dgh(X,Y) \geq \db(\dgm_k^{\vr}(X), \dgm_k^{\vr}(Y))$ is tight. This is depicted by the following examples: for $\delta \geq 0$, let $X_\delta = \{a,b \} $ be the metric space with $d(a,b) = 1+ \delta$. Clearly, $\dgh(X_\delta, X_0) = \frac{\delta}{2}$. We also observe that $\dgm_0^{\vr}(X_\delta) = \lc [0, 1+ \delta], [0, \infty) \rc$. Thus, $\dgm_0^{\vr}(X_0) = \lc [0,1], [0, \infty) \rc$, and we obtain that $\db(\dgm_0^{\vr}(X_\delta), \dgm_0^{\vr}(X_0)) = \delta $. Therefore, we have $$2 \cdot \dgh(X_\delta, X_0) = \db(\dgm_0^{\vr}(X_\delta), \dgm_0^{\vr}(X_0)).$$
Now consider the metric spaces $X$ and $X_0$, where $X$ is the one point metric space and $X_0$ is as defined above. Then, $\dgh(X,X_0) = \frac{\diam(X_0)}{2} = \frac{1}{2}$. We have $\dgm_0^{\vr}(X) = \lc [0, \infty) \rc $ and $\dgm_0^{\vr}(X_0) = \lc [0,1], [0, \infty) \rc $. Thus, we have $\db(\dgm_0^{\vr}(X), \dgm_0^{\vr}(X_0)) = \frac{1}{2}$. Therefore, we obtain 
$$2 \cdot \dgh(X_0, X) > \db(\dgm_0^{\vr}(X_0), \dgm_0^{\vr}(X)).$$

\section{Applications of Persistent Homology} \label{sec:applications}

In this section, we first describe two problems in neuroscience that have been studied using persistent homology. We state their problem definitions, the experimental procedures that generate a finite data set, the method used to construct an abstract simplicial complex from the data set, and finally the results obtained using persistent homology. In the third and fourth subsections, we describe further applications of persistent homology to biology as well as to other areas.

\subsection{A Topological Paradigm for Hippocampal Spatial Map Formation using Persistent Homology}

This subsection describes article \cite{DMFC12} of the same title. The problem is that of identifying the topological features of an environment using the hippocampal activity of a rat moving in that environment. In every animal, the hippocampus is the region of the brain responsible for creating a mental map of the animal's environment. This mental map is made possible by activity of the neurons in the hippocampus called \emph{place cells}. As an animal explores a given environment, different place cells fire a series of action potentials in different, discrete regions of the environment. Each region, referred to as that cell's \emph{place field}, is defined by the pattern of neuronal firing, most intense at the center and attenuated towards the edges of the field. The cell remains silent when the animal is outside of the cell's place field. Experiments on rats suggest that the information contained in place cell firing patterns encodes spatial navigation routes and somehow represents the spatial environment \cite{Brown7411,McNaughton1983,ZGMS98}. Now, suppose that spatial location is the primary determinant of each place cell's firing. Then, co-firing of several place cells indicates that the corresponding place fields overlap, See Figure \ref{fig:place_cells}. Thus, the mental map formed by co-firing will be based on the properties of connectivity, adjacency and containment of place fields, and therefore will be a topological map of the environment. 

\begin{figure}
\centering
\scalebox{0.6}{\includegraphics{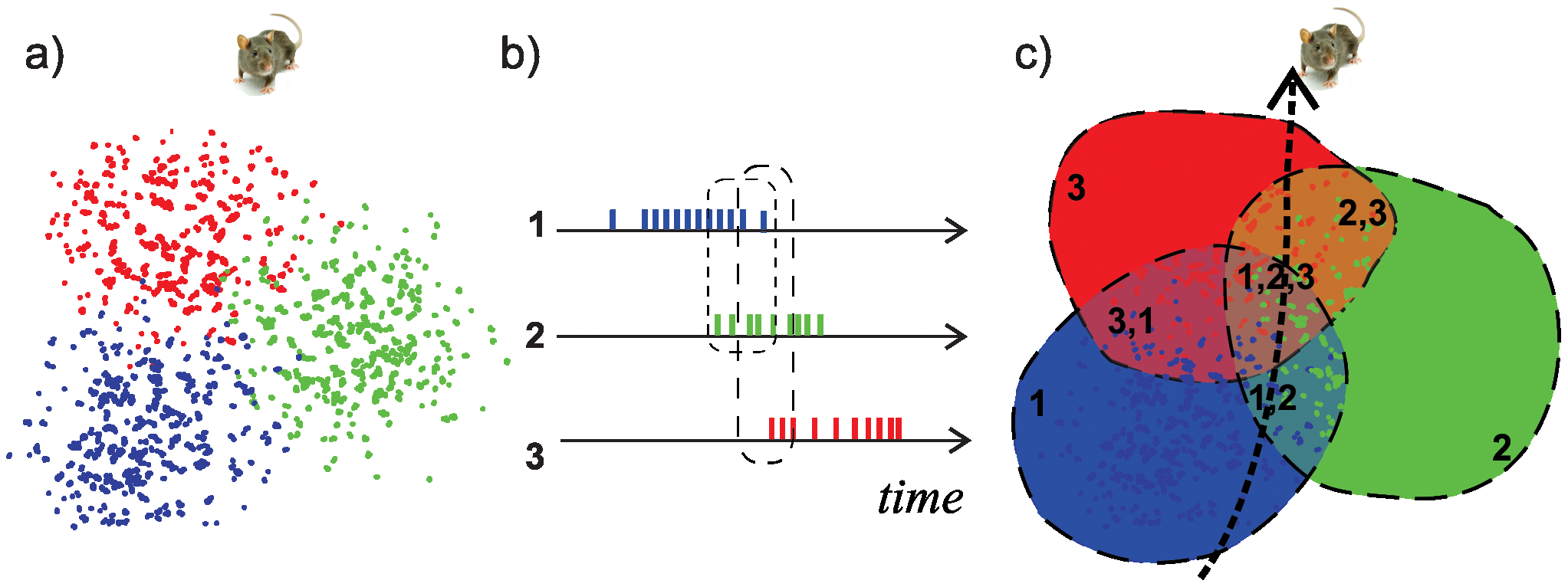}}
\caption{\small This figure is from \cite{DMFC12}. a) As a rat explores an environment, various place cells fire in spatially discrete locations. Here, three place fields are depicted as they might arise from spike trains from three place cells. b) A representation of spike trains fired from three different place cells as a rat explores the environment. We note that there is co-firing. c) The place fields derived from the three place cells in b); the co-firing patterns indicates areas of overlap of the place fields. When the rat makes a straightforward trajectory through an explored environment, different place cells will be activated and their place fields can overlap.} \label{fig:place_cells}
\end{figure}

A basic theorem of algebraic topology is the so called \emph{Nerve Theorem} \cite{Hatcherbook} which we paraphrase here as: if a space $X$ is covered with a sufficient number of regions, then it is possible to reconstruct the topology of $X$ using the intersection information of the regions. This theorem and the assumption that the place fields cover the environment leads to the hypothesis that the overlaps between the place fields, as represented by temporal overlap of spike trains (an ordered list of times at which a place cell fires) provide a connectivity map that retains the topological features of the environment. Thus, the authors of \cite{DMFC12} investigate whether a topological connectivity map can be effectively and reliably derived from neuronal spiking patterns using computational tools in the field of algebraic topology.  

\begin{figure}
\centering
\scalebox{0.7}{\includegraphics{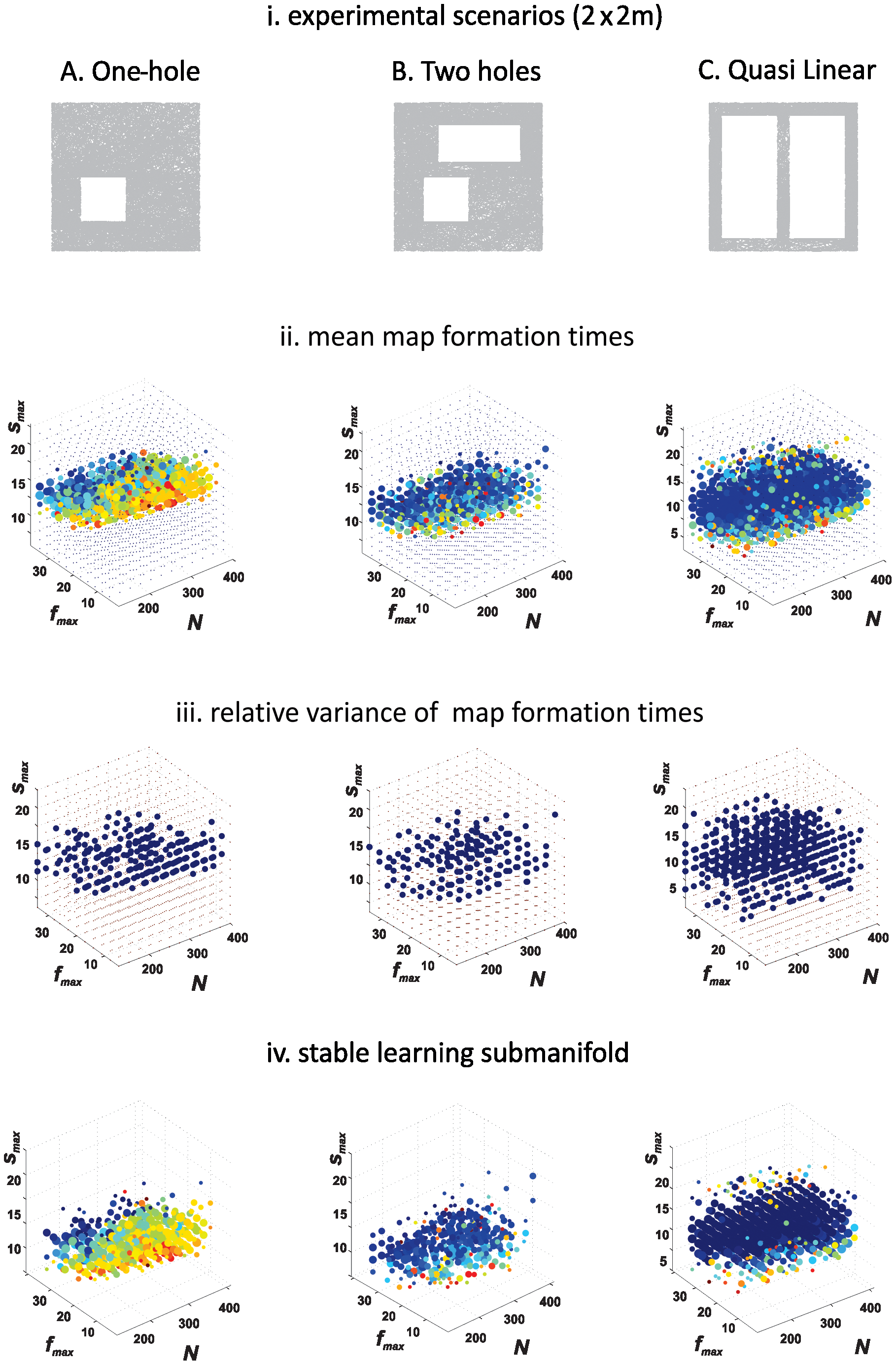}}
\caption{\small This figure is from \cite{DMFC12}. The top row depicts three experimental configurations, each $2 \times 2$ square, for the computational simulations. The configurations B and C are topologically equivalent but geometrically different. The dense network of gray lines represents the simulated trajectories. \emph{Second row}: Point cloud approximations that reveal mean map formation times for each space configuration. Each dot represents a ``hippocampal" state as defined by the three parameters $(\overline{f}, \overline{s}, N)$; the size of the dot reflects the proportion of trials in which a given set of parameters produced the correct outcome; the color of the dot is the mean time over ten simulations. If, for example, one set of parameters produced the correct topological information in $6$ out of $10$ trials, the dot will be $60 \%$ of the size of the largest dot, and the color will reflect the mean map formation time for the correct trials. (Blue represents success within the first $25 \%$ of the total time; green within the first $50 \%$, yellow-orange within the first $75 \%$, and red means success took nearly the whole time period. The maximal observed time was $4.3$ minutes for configuration A, $11.7$ minutes for B, and $9.3$ minutes for C.) We note that the third scenario (C) contains a preponderance of blue dots, indicating that the majority of hippocampal states easily mapped this environment. This is because the two holes are so large that a rat is virtually forced into a straight-line trajectory. We refer the reader to the original paper \cite{DMFC12} for a description of the third and fourth rows.} 
\label{fig:mean_map}
\end{figure}

We now briefly describe the details of the experiments performed in \cite{DMFC12}. The authors simulated map formation times (minimal time required to produce the correct topological signature of an environment) using different place cell parameters and three separate planar $2 \times 2$ meter areas with $1$ or $2$ holes. The place cell parameters are the firing rates, the place field sizes and the number of place cells. The firing rates and the place field sizes are described by $\log$-normal distributions, with $\overline{f}$ and $\overline{s}$ being the respective peak values, and $1.2 \overline{f}$ and $1.7 \overline{s}$ being the respective standard deviations. The question was for which parameters the place cell spiking signals would be able to produce a temporal simplicial complex with the correct number of topological loops, or Betti numbers, in every dimension. The authors probed ten distributions of firing rates, with $\overline{f}$ ranging from $2$ to $40$ Hz, and ten distributions of place field sizes, with $\overline{s}$ ranging from $10$ to $90$ cm. The number of place cells varied independently from $N = 50$ to $N =400$. In each case, the centers of the place fields were scattered randomly and uniformly over the environment. For each combination of the parameters, $\overline{f}, \overline{s}$ and $N$, the computation was repeated $10$ times, through which the authors computed the average time $T_{min}$ required for the emergence of correct topological features for each specific choice of ensemble parameters $\overline{f}, \overline{s}$ and $N$. The authors fix the simulated trajectory, but choose a new set of place field centers for each set of $\overline{f}, \overline{s}, N$ for each repetition. 

We now describe the mechanism of generating a filtered simplicial complex from the experimental data. Let $\{PF_1, PF_2, \ldots, PF_n \} $, $n \in \N$ be a set of place fields with specified shapes and locations. A simplicial complex $S$ with vertex set $\{v_1, v_2, \ldots, v_n \}$, one $v_i$ for each place field, can be constructed as follows: given $k \in \N$, a simplex $[v_{i_1}, v_{i_2}, \ldots, v_{i_k}] \in S $ if $PF_{i_1} \cap PF_{i_2} \cap \ldots \cap PF_{i_k} \neq \emptyset$. We recall that this coincides with the \emph{\v{C}ech} complex, also known as the \emph{nerve complex}. The Nerve Theorem \cite{Hatcherbook} states that if there is a space $X$ such that $X = \cup_{i=1}^n PF_i$ and each finite intersection of the place fields is contractible, then under fairly general conditions, the nerve complex $S$ has the same homotopy type as the underlying space $X$, and so the topological invariants computed from $S$ will agree with those corresponding to $X$. We saw that the experimental data does not consist of the place fields, but of the spike trains of the place cells. Thus, an overlap of the place fields is identified by co-firing of the corresponding place cells. Let $\{c_1, c_2, \ldots, c_n \}$ denote the place cells corresponding to the place fields $\{PF_1, \ldots, PF_n \}$ respectively, and let $\{s_1, s_2, \ldots, s_n \}$ denote the corresponding spike trains. We recall that for $i \in [1:n]$, a spike train $s_i$ is an ordered list of times at which the place cell $c_i$ fires. We fix an $\eps > 0$ and an $m \in \N$. We define a filtered simplicial complex as follows: given a simplex $\sigma = [i_1, i_2, \ldots, i_k]$, we define a function $f$ on $\sigma$ as 
$$f(\sigma) = \min \{t~:~\min_{j \in [1, \ldots,k]}|s_{i_j} \cap [t-\eps, t+ \eps] | \geq m \}. $$  
By definition, we have $f(\sigma) \leq f(\tau)$ if $\sigma \subseteq \tau$. Thus, we start with an empty simplicial complex, and then add simplices to this complex, according to the values of the simplices on the function $f$. The homology functor $H_k(\ast, \mathbb{Z}_2)$ is applied to this filtered simplicial complex, for $k =0,1$. For each $k = 0,1$, this produces a persistence vector space, and thus a barcode. Barcodes are used to determine the first two Betti numbers, $\beta_0$ and $\beta_1$. We recall that $\beta_0$ tells the number of connected components, and $\beta_1$ tells the number of $1$-dimensional holes. The software used to analyze the data is jPLEX \cite{sexton_jplex_2008}, a collection of MATLAB functions for computational topology that implements the concepts described above.

The results obtained in \cite{DMFC12} and their interpretation are depicted in Figure \ref{fig:mean_map}. The authors observed that the place cell parameters of firing rate and place field size for which a reliable topological map of the environment is produced correspond well with experimentally observed place cell firing rates and place field sizes. Thus, the fact that these parameters fall into the biological range lends support to this topological paradigm. 

\subsection{Topological Analysis of Population Activity in Visual Cortex}

This subsection describes article \cite{SMISCR08} of the same title. This work studies some basic aspects of the patterns of activity in the primary visual cortex (V1) evoked by natural images and during spontaneous activity. The authors focus on a topological characterization of population activity in visual cortex. The reason behind this approach is the following: it has been observed that spontaneous cortical states tend to reproduce the patterns evoked by oriented stimuli \cite{Kenet03}. Now, if cortical activity is restricted to patterns evoked by an oriented stimulus, then considering that orientation is a circular variable, this leads to the hypothesis that the activity patterns of the cortical cells must have a topological structure equivalent to that of a circle. This implies that the basic question about the structure of the cortical activity data is topological in nature. This work offers the first estimate of the underlying topological structure of V1 activity.

We now describe the experimental procedures adopted in \cite{SMISCR08}. The authors first validate their method on simulated data by recovering the topological structure of data sets where the ``ground truth" is known. The validation is done for a circle as well as torus. We refer the readers to the original paper \cite{SMISCR08} for details of the validation methods. The experimental studies were performed on three old-world monkeys (\emph{Macaca fascicularis}), See Figure \ref{fig:monkey_exper}. The database considered in this study was obtained using micro-machined electrode arrays consisting of a square grid of $10 \times 10$ electrodes $1.5$ mm long. The distance between neighboring electrodes was $400$ $\mu$m. Spike sorting was performed online using principal component analysis on the waveform shapes. In the spontaneous condition, the eyes were covered. The stimuli in the evoked condition were image sequences generated by digitally sampling commercially available videotapes in VHS format. The selected movies included both man-made and natural landscape scenes, and $6$ segments of $30$ seconds duration were shown.

\begin{figure}
\centering
\scalebox{0.25}{\includegraphics{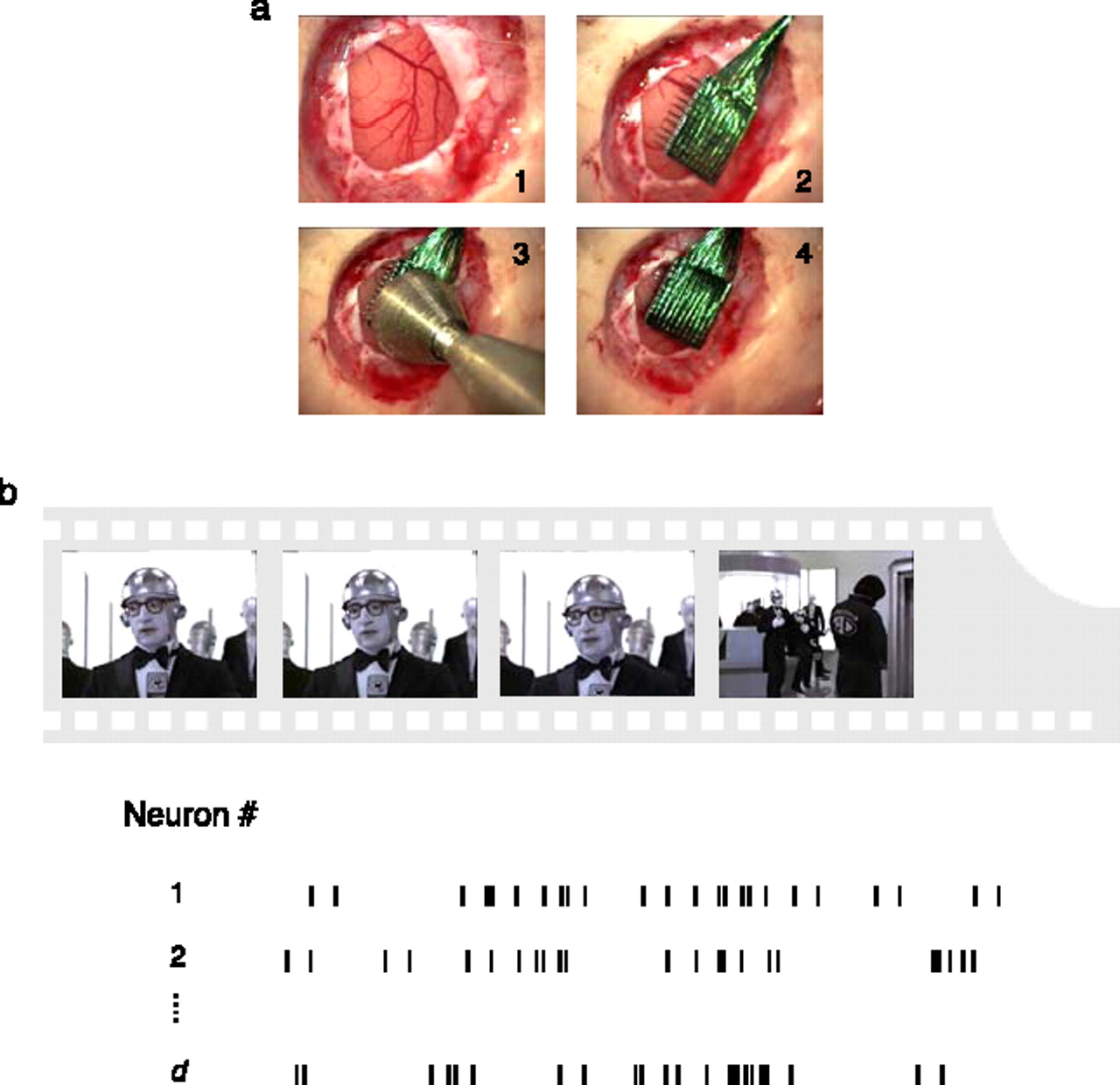}}
\caption{\small This figure is from \cite{SMISCR08}. \textbf{Experimental recordings in primary visual cortex.} a) Insertion sequence of a multi-electrode array into primary visual cortex. b) Natural image sequences, sampled from commercial movies, were used to stimulate all receptive fields of neurons isolated by the array. In the spontaneous condition, activity was recorded while both eyes were covered.} \label{fig:monkey_exper}
\end{figure}

We now describe how the data points were generated from the experiment described above. The preparation of the data points for both the spontaneous and driven activity during natural image simulation was identical. After spike-sorting signals from each electrode, the authors sub-selected a group of $5$ neurons that showed the highest firing rates. Then, a point cloud was generated by binning spikes in $50$ ms windows. The spontaneous and evoked activity segments were collected in lengths of 10 s each. Thus, each of these segments contain $200$ points living in $\mathbb{R}^5$, each neuron corresponding to a dimension. The statistical package PLEX was used with a weak witness complex construction which will be explained in the next paragraph. PLEX is a MATLAB collection of functions for computational topology. The authors recorded the maximal length of persistence intervals in the $1$-dimensional and $2$-dimensional barcodes.

We now describe the weak witness complex construction \cite{SMISCR08}. Given a finite metric space $(X,d_X)$, a set of points $L \subset X$ called the \emph{landmark set}, and $\epsilon > 0$, a point $x \in X$ is called an $\epsilon$-witness for a $k+1$-tuple $\{l_0, l_1, \ldots, l_k \} $ of points in $L$ if $\max_i d_X(x,l_i) \leq \epsilon + m_x$, where $m_x$ denotes the $k+1$ smallest value of $d_X(x,l)$ as $l$ varies over all of $L$. Now, a simplicial complex $W_\epsilon(X,L)$ is associated to $X, L$ and $\epsilon$ by fixing the vertex set of $W_\epsilon(X,L)$ to be $L$, and declaring that a collection $\{l_0, l_1, \ldots, l_k \} $ spans a $k$-simplex in $W_\epsilon(X,L)$ if and only if there is an $\epsilon$-witness in $X$ for the collection $\{l_0, l_1, \ldots, l_k \} $ and for all its faces. Clearly, if there is an $\epsilon$-witness for the simplex $\sigma = \{l_0, l_1, \ldots, l_k \} $, then there is an $\epsilon'$-witness for $\sigma$, $\epsilon' \geq \epsilon$. Thus, we obtain that for $\epsilon \leq \epsilon'$, $W_\epsilon(X,L) \subseteq W_{\epsilon'}(X,L)$ and this results in a filtered simplicial complex. In \cite{SMISCR08}, out of the $200$ data points in $\mathbb{R}^5$, a landmark set of $35$ points is chosen by the max-min procedure as follows: first a random point, say $x_1$ from $X$ is picked. Then, the point $x_2$ is chosen such that $d_X(x_1,x_2)$ is maximized. The point $x_3$ is chosen such that $d_X(x_3, \{x_1,x_2 \})$ is maximized, and so on. The weak witness construction was used because, unlike the Vietoris-Rips simplicial complexes, the construction of weak witness simplicial complexes for large data sets is much more computationally tractable.

\begin{figure}
\centering
\scalebox{0.3}{\includegraphics{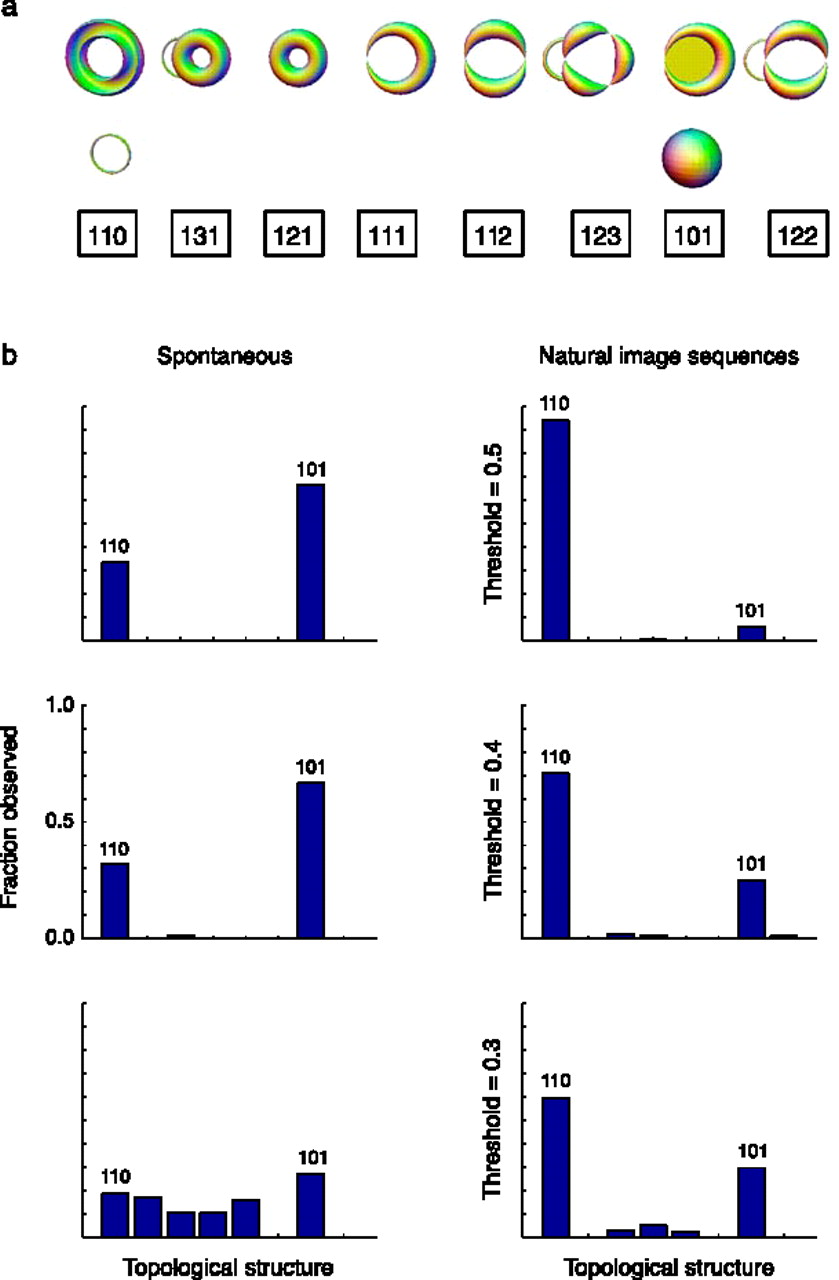}}
\caption{\small This figure is from \cite{SMISCR08}. \textbf{Estimation of topological structure in driven and spontaneous conditions.} a) Ordering of topological signatures observed in the experiments. Each triplet $(b_0, b_1, b_2)$ is shown along an illustration of objects consistent with these signature. b) Distribution of topological signatures in the spontaneous and natural image stimulation conditions pooled across the three experiments performed. Each row correspond to signatures with a minimum interval length (denoted as the threshold) expressed as a fraction of the covering radius of the data cloud.} \label{fig:Exper_results}
\end{figure}

We now describe the results obtained in \cite{SMISCR08}. In Figure \ref{fig:Exper_results}, different topological signatures observed in $10$ s segments of the data labeled by the first three Betti numbers $(b_0, b_1, b_2)$ are illustrated. Each row of Figure \ref{fig:Exper_results} represents a different ``threshold'' for the length of the interval of the signature (in the barcode) as a fraction of the covering radius of the data. The covering radius is defined as $R_0 = \max_{x \in X} \min_{l \in L}d_X(x,l)$, where $X$ is the data set and $L$ is the set of landmarks. Larger thresholds represent instances where the signature was long-lived and likely to represent a salient feature of the data. 

\subsection{Further Applications to Biology}

This paragraph describes some more applications of persistent homology to neuroscience. In \cite{Spreemann18}, the authors propose a method based on persistent homology to automatically classify neuronal network dynamics using topological features of spaces built from spike-train distances. The dynamics of a neuronal network are believed to be indicative of the computations it can perform, and thus, understanding the neuronal network dynamics enables understanding of how neuronal networks perform computations and process information. The paper \cite{CDM18} is an extension to \cite{DMFC12}, wherein the authors use the concept of zig-zag persistent homology \cite{Carlsson-Silva2010,CSM09} to account for the possibility of forgetting information in the model for memory. The results obtained in \cite{CDM18} show that in order to achieve the best possible results in ``learning'' an arena, the rodent needs a balance between remembering and forgetting information. These results are in accordance with recent findings in neuroscience, where it has been proposed that forgetting is an important step in the learning process. The work by Giusti et.~al.~in \cite{GGB16} explores the method of persistent homology over the traditional graph-theoretic methods, for understanding neural data. 

In \cite{XW14}, persistent homology is used for the first time for protein characterization, identification and classification. The authors extracted molecular topological fingerprints based on the persistence of molecular topological invariants. In \cite{ESR16}, persistent homology is used to characterize the complex structure of chromatin inside cell nucleus. The authors apply persistent homology to human cell line data and show how this method captures complex multiscale folding methods.

In \cite{Chan18566}, persistent homology is used to study evolutionary events. The authors consider a set of genomes and calculate the genetic distance between each pair of sequences. Using these distances, they calculate the homology groups across all genetic distances $\eps$ in different dimensions. They observe that the zero-dimensional homology provides information about vertical evolution, i.e.~at a particular $\eps$, the Betti number $b_0$ represents the number of different strains or subclades. The one-dimensional homology provides information about horizontal evolution since reticulate events (merging of different clades to form a new hybrid lineage) are represented by loops in phylogenetic networks. Some examples of reticulate events include recombination and reassortment of genomes. The genomic datasets used are those of influenza strains, HIV, rabies, dengue, flaviviruses, West Nile virus and Newcastle virus. In a follow-up paper \cite{CLR16}, persistent homology is used to study the specific evolutionary event of recombination. In \cite{Chan18566}, the relation between persistent homology and explicit evolutionary histories incorporating recombination events was not studied. Therefore, in \cite{CLR16}, persistent homology is applied on appropriate genomic sets in order to characterize the genomic regions where recombination takes place and identify the gametes involved in particular recombination events. The persistent homology barcodes derived from each of these sets are structured as a ``barcode ensemble'' where each bar captures a recombination event. A software called TARGet is developed that generates a graph in polynomial time, capturing ensembles of minimal recombination histories. The evolutionary event of recombination has been further studied in \cite{Lesnick2018QuantifyingGI} where the authors introduce ``novelty profiles'' of evolutionary histories. The novelty profile of an evolutionary history is a list of $k$ monotonically decreasing numbers, where $k$ is the number of recombination events in the history and each number roughly measures the contribution every recombination makes to the genetic diversity in the population. Persistent homology of sampled data is used to obtain information about a novelty profile. The authors of \cite{Lesnick2018QuantifyingGI} provide mathematical foundation for several works that have used persistent homology to study recombination. Some other articles showing the use of persistent homology for studying recombination events are \cite{ER14, CRELR16, ER16}.

In a different direction, another topological method for studying finite metric spaces is \emph{Mapper} \cite{mapper}. It is a computational method for extracting simple descriptions of high dimensional data sets in the form of simplicial complexes. This method has been widely used for analysis of biological data sets as seen in \cite{Nicolau7265,pathways,Lee2017SpatiotemporalGA, cecco15,torres16,OLIN2018,saggar18,frattini18,Rucco2014UsingTD,Bartlett2012,Phinyomark2018,Kyeong15,Selinger2014,CAMARA201795,phinyomark17,sgouralis17}. 

\subsection{Applications to Other Domains}

Persistent homology has also been used for shape classification. The authors of \cite{CCGMO09} use persistent homology to identify signatures of finite metric spaces that are stable under the Gromov-Hausdorff distance. The signatures are nothing but metric invariants obtained using persistent homology along with attributes of the metric spaces like diameter and eccentricity. These signatures are computed and then used to measure the degree of dissimilarity of a pair of metric spaces. The authors adapt this method to compare shapes, by first uniformly sampling points from each shape to generate a finite metric space and then comparing the finite metric spaces using the identified signatures. 

In this paragraph, we provide two examples where persistent homology has been used for studying chemical compounds. In \cite{XFTW14}, persistent homology is used for studying fullerenes, which are special molecules consisting of only carbon atoms. Here, the point cloud is given by the atoms of the fullerenes, and a Vietoris-Rips filtration is constructed by the usual process of assigning radii to the point cloud. The authors thus study the stability of the fullerene molecules by observing that the total curvature energies of the fullerene isomers can be well represented with the lengths of their long-lived Betti $2$-bars. In \cite{Hess18}, persistent homology is used to build a descriptor for identifying and comparing zeolites, according to their pore shapes. Zeolites are nanoporous materials made of silica. The authors performed high-throughput screening of zeolites based on this descriptor and identified best zeolites for methane storage and carbon capture applications. The results obtained in \cite{Hess18} match the existing results on top-performing zeolites for these applications.

\section{Software Packages for Persistent Homology} \label{sec:software}

There are various open source softwares available for computing persistent homology. These are available in R, Python, C$++$, Java as well as in MATLAB. The softwares are \textbf{Perseus} \cite{perseus}, \textbf{PHAT} \cite{phat}, \textbf{DIPHA} \cite{dipha}, \textbf{CTL} \cite{ctl}, \textbf{Ripser} \cite{ripser}, \textbf{TDA} \cite{tda}, \textbf{javaPlex} \cite{Javaplex}, \textbf{Dionysus} \cite{dionysus}, \textbf{Gudhi} \cite{gudhi}, \textbf{TDAstats} \cite{TDAstats}, \textbf{Scikit-TDA} \cite{scikittda2019} and \textbf{the Topology Toolkit} \cite{toolkit}.

\bibliographystyle{alpha}
\bibliography{references}
\end{document}